\documentclass[xcolor=dvipsnames,svgnames,table,reqno,11pt]{amsart}

\usepackage[left=1.5in, right=1.5in, bottom=1.25in, height=8in]
{geometry}

\input xy
\xyoption{all}
\usepackage{accents}
\usepackage{epsfig}
\usepackage{xcolor}
\usepackage{amsthm}
\usepackage{amssymb}
\usepackage{bm} %\usepackage{bbm}
\usepackage{amsmath}
\usepackage{amscd}
\usepackage{amsopn}
\usepackage{graphicx}
\usepackage{tikz}
\usetikzlibrary{shapes,shadows,arrows}
\usepackage{tikz-cd}

\usepackage{xspace}
\usepackage{hhline}
\usepackage{easybmat}
\usepackage{caption}   
\usepackage{relsize}

\usepackage{url}
\usepackage{enumitem, hyperref}\hypersetup{colorlinks}

\usepackage{stmaryrd}

\usepackage[T1]{fontenc}

\colorlet{purpleB70}{blue!70!red}

\colorlet{orangeR65}{red!65!yellow}

\definecolor{red2}{HTML}{d41173}

\definecolor{neongreen}{HTML}{1bf702}

\definecolor{radicalred}{HTML}{FF355E}

\definecolor{denim}{HTML}{1560BD}

\definecolor{darkcyan}{rgb}{0.0, 0.55, 0.55}

\definecolor{cilek}{HTML}{FF43A4}

\definecolor{mor}{HTML}{9F00C5}

\definecolor{phlox}{rgb}{0.87, 0.0, 1.0}

\definecolor{fluorescentpink}{HTML}{FF1493}

\definecolor{napiergreen}{rgb}{0.16, 0.5, 0.0}

\definecolor{kellygreen}{rgb}{0.3, 0.73, 0.09}

\definecolor{parisgreen}{HTML}{ 50C878 }

\definecolor{palatinateblue}{rgb}{0.15, 0.23, 0.89}

\definecolor{ceruleanblue}{rgb}{0.16, 0.32, 0.75}

\definecolor{brandeisblue}{rgb}{0.0, 0.44, 1.0}

\definecolor{KLMblue}{HTML}{0FC0FC}

\definecolor{cinnamon}{rgb}{0.82, 0.41, 0.12}

\definecolor{darkorange}{rgb}{1.0, 0.55, 0.0}

\definecolor{darktangerine}{rgb}{1.0, 0.66, 0.07}

\definecolor{deepcarrotorange}{rgb}{0.91, 0.41, 0.17}

\definecolor{internationalorange}{HTML}{FF4F00}

\definecolor{persimmon}{HTML}{EC5800}

\definecolor{pumpkin}{HTML}{FF7518}

\definecolor{darkred}{rgb}{1,0,0} %can change the intensity in [0,1]
\definecolor{darkgreen}{rgb}{0,0.7,0}
\definecolor{darkblue}{rgb}{0,0,1}

\hypersetup{colorlinks,
linkcolor=darkblue,
filecolor=darkgreen,
urlcolor=darkred,
citecolor=darkgreen}

\makeatletter
\def\reflb#1#2{\begingroup
    #2%
    \def\@currentlabel{#2}%
    \phantomsection\label{#1}\endgroup
}
\makeatother

\numberwithin{equation}{section}
\newtheorem{Theorem}{Theorem}
\numberwithin{Theorem}{section}

\newtheorem{TheoremX}{Theorem}

\newtheorem   {Lemma}[Theorem]{Lemma}
\newtheorem   {Claim}[Theorem]{Claim}

\newtheorem   {Proposition}[Theorem]{Proposition}
\newtheorem   {Corollary}[Theorem]{Corollary}
\theoremstyle {definition}

\theoremstyle {remark}
\newtheorem   {Remark}[Theorem]{Remark}
\newtheorem   {Example}[Theorem]{Example}

\def    \eps    {\epsilon}

\newcommand{\CA}{{\mathcal A}}

\newcommand{\DD}{{\mathcal D}}

\newcommand{\CI}{{\mathcal I}}

\newcommand{\CS}{{\mathcal S}}
\newcommand{\CV}{{\mathcal V}}

\newcommand{\id}{{\mathit id}}

\newcommand{\const}{{\mathit const}}

\newcommand{\charr}{{\mathit char}\,}

\newcommand{\ty}{\tilde{y}}

\newcommand{\tx}{\tilde{x}}

\newcommand{\tz}{\tilde{z}}
\newcommand{\cz}{\check{z}}
\newcommand{\hz}{\hat{z}}

\newcommand{\tH}{\tilde{H}}

\newcommand{\A}{{\mathcal A}}

\newcommand{\Bb}{{\mathcal B}}
\newcommand{\CB}{{\mathcal B}}

\newcommand{\Pp}{{\mathcal P}}
\newcommand{\PP}{{\mathcal P}}

\newcommand{\Ss}{{\mathcal S}}

\def    \F      {{\mathbb F}}

\def    \R      {{\mathbb R}}

\def    \Z      {{\mathbb Z}}
\def    \N      {{\mathbb N}}
\def    \Q      {{\mathbb Q}}

\def    \CP     {{\mathbb C}{\mathbb P}}

\def    \12     {{\frac{1}{2}}}

\def    \p      {\partial}

\def    \rk     {\operatorname{rk}}

\def    \SH     {\operatorname{SH}}
\def    \COH    {\operatorname{CH}}

\def    \Tor    {\operatorname{Tor}}

\def    \HF     {\operatorname{HF}}

\def    \H      {\operatorname{H}}

\def    \Tor    {\operatorname{Tor}}

\def    \CF      {\operatorname{CF}}

\def    \hmu   {\operatorname{\hat{\mu}}}

\def    \End    {\operatorname{End}}
\def    \Beg    {\operatorname{Beg}}

\def    \cf    {\operatorname{c}}

\def    \TSp     {\widetilde{\operatorname{Sp}}}

\newcommand \beg   {\operatorname{Beg}}
\newcommand \en   {\operatorname{End}}

\newcommand \Cbar {C_{\scriptscriptstyle{bar}}}

\newcommand   \slope {\operatorname{\mathit{slope}}}
\newcommand   \rmax {r_{\max}}
\newcommand   \WW {\widehat{W}}
\newcommand   \UU {\widehat{U}}
\newcommand   \VV {\widehat{V}}

\newcommand   \cfp   {\operatorname{c}^{(p)}}

\begin{document}

%%%%%%%%%%%%%%%%%%%%%%%%%%%%%%
%   TEXT FORMATTING

\setlength{\smallskipamount}{6pt}
\setlength{\medskipamount}{10pt}
\setlength{\bigskipamount}{16pt}

%%%%%%%%%%%%%%%%%%%%%%%%%%

%%%%%%%%%%%%%%%%%%%%%%%%%%

%%%%%%%%%%%           BEGINNING OF  TEXT

%%%%%%%%%%%%%%%%%%%%%%%%%%

\title [Filtered Symplectic Homology and Closed Reeb Orbits]{Filtered
  Symplectic Homology \\ and Closed Reeb Orbits}

\author[Erman \c C\. inel\. i]{Erman \c C\. inel\. i}
\author[Viktor Ginzburg]{Viktor L. Ginzburg}
\author[Ba\c sak G\"urel]{Ba\c sak Z. G\"urel}

\address{E\c C: ETH Z\"urich, R\"amistrasse 101, 8092 Z\"urich,
  Switzerland}
\email{erman.cineli@math.eth.ch}

\address{VG: Department of Mathematics, UC Santa Cruz, Santa Cruz, CA
  95064, USA} \email{ginzburg@ucsc.edu}

\address{BG: Department of Mathematics, University of Central Florida,
  Orlando, FL 32816, USA} \email{basak.gurel@ucf.edu}



\date{\today} 

\thanks{The work is partially supported by the NSF grants DMS-2304207
  (BG) and DMS-2304206 (VG), the Simons Foundation grants 855299 (BG)
  and MP-TSM-00002529 (VG), and the ERC Starting Grant 851701 via a
  postdoctoral fellowship (E\c{C})}

\begin{abstract}
  We further explore connections between the symplectic homology
  persistence module and the properties of closed Reeb orbits for
  star-shaped domains in higher dimensions. Our first result is that
  the sequence of $S^1$-equivariant spectral invariants over a field
  of positive characteristic is bounded from above, in contrast with
  the case of characteristic zero. We also prove that the dimension of
  the filtered symplectic homology is bounded as a function of the
  action whenever the flow is a pseudo-rotation, i.e., it has finitely
  many prime closed orbits. Finally, we show that a non-degenerate
  Reeb flow has infinitely many prime closed orbits whenever it has
  one closed orbit with negative mean index.
\end{abstract}

\maketitle

\vspace{-0.2in}

%{\small \tableofcontents }

\tableofcontents

\section{Introduction and main results}
\label{sec:intro+results}

\subsection{Introduction}
\label{sec:intro}
In this paper we continue studying connections between the symplectic
homology persistence module and the dynamics of the underlying Reeb
flow, focusing on the existence and properties of closed Reeb orbits, 
 mainly in higher dimensions. While formally independent, the
paper can be viewed as a conceptual follow up to \cite{CGG:Reeb-HZ}
where we proved variants of the multiplicity (aka Ekeland's) and
Hofer--Zehnder conjectures for a certain class of Reeb flows.

Namely, in that paper we showed that a dynamically convex Reeb flow on
the boundary of a star-shaped domain $W\subset \R^{2n}$ must have at
least $n$ prime closed orbits. Furthermore, the flow has either
exactly $n$ or infinitely many such orbits, provided that it is
non-degenerate and $W$ is centrally symmetric. The first part is the
multiplicity conjecture often attributed to Ekeland; cf.\
\cite{Ek}. The second part -- either exactly $n$ or infinitely many
prime closed orbits -- is a variant of the Hofer--Zehnder conjecture;
cf.\ \cite{HZ} where it is stated for Hamiltonian diffeomorphisms of
$\CP^n$.  A key new step in the proofs is the statement that the
filtered symplectic homology of $W$ over any field is one-dimensional
for every positive action threshold whenever the flow is a dynamically
convex pseudo-rotation. (Recall that a Reeb flow is called a
pseudo-rotation if it has finitely many prime closed orbits.)

Here we further explore this type of interplay between symplectic
homology and periodic orbits. The simplest and most-explored
connection is at the level of spectral invariants over $\Q$.  The
spectral invariants over fields of positive characteristic are much
less used and understood. Our first result, Theorem
\ref{thm:spectral}, is that the sequence of $S^1$-equivariant spectral
invariants over a field of characteristic $p>0$ is bounded from above
for every Liouville domain in $\R^{2n}$. For instance, for an
ellipsoid, the upper bound is the action of the first $p$-iterated
orbit.  This fact is rather unexpected. In the case of zero
characteristic, the spectral invariants grow roughly linearly. It also
constitutes a serious difficulty in removing the symmetry condition in
the proof of the Hofer--Zehnder conjecture in \cite{CGG:Reeb-HZ}. One
of our goals here is to initiate the study of positive characteristic
$S^1$-equivariant spectral invariants.

The next three results are related, directly or indirectly, to Reeb
pseudo-rotations. Hypothetically, perhaps under minor additional
assumptions, every Reeb pseudo-rotation is dynamically convex,
non-degenerate and has exactly $n$ prime closed orbits. This
conjecture is wide open and we treat it as a whole circle of different
questions. One implication of dynamical convexity is that all closed
Reeb orbits have mean index greater than or equal to two. Our next
result, Theorem \ref{thm:negative}, shows that a non-degenerate
pseudo-rotation cannot have closed Reeb orbits with negative mean
index. In other words, a non-degenerate Reeb flow on the boundary of a
star-shaped domain must have infinitely many prime closed orbits
whenever it has a closed orbit with negative mean index. The proof
again is based on the analysis of the symplectic homology persistence
module.

Then we show that for a Reeb pseudo-rotation on the boundary of a
Liouville domain in $\R^{2n}$, the dimension of the filtered
symplectic homology is bounded from above as a function of the
action. This is Theorem \ref{thm:bounded}, which is also the first
step in the proof of Theorem \ref{thm:negative}. As we have already
pointed out, this dimension is exactly one for positive actions when
the pseudo-rotation is dynamically convex and the domain is
star-shaped; \cite[Thm.\ C]{CGG:Reeb-HZ}. (The latter condition is
probably inessential.)  Theorem \ref{thm:sh=1} shows that this
requirement on the dimension has strong consequences for the
properties of closed Reeb orbits, refining and streamlining some of
the results from \cite{CGG:Reeb-HZ}. Informally speaking, the main
point of the theorem is that the local homology spaces of the orbits
behave as if the orbits were non-degenerate and, assuming
non-degeneracy, the filtered symplectic homology is similar to that of
an irrational ellipsoid.  However, we do not know if the dimension
requirement alone guarantees that the flow is a pseudo-rotation even
when it is non-degenerate and dynamically convex.

The difficulty in many dynamics questions for Reeb flows in dimensions
greater than three lies in the fact that there is no clear-cut
connection between symplectic homology and the number of prime closed
Reeb orbits. Intuitively, a larger number of orbits should be
reflected by a larger filtered symplectic homology persistence
module. However, there is no obvious way to make this statement
precise. In the multiplicity and Hofer--Zehnder conjectures, we use
the dimension as a function of the action to be a measure of size, but
even that connection is far from direct.

%\begin{Remark}
The interplay between the properties of the symplectic homology and
the dynamics of the Reeb flow goes beyond periodic orbits. For
instance, the barcode growth at a small scale, which is yet another
way to measure the size of a persistence module, reflects such
disparate dynamics features as positive entropy on one side and
complete integrability on the other. We refer the reader to
\cite{CGGM:LN} for a survey of results in that direction and further
references.
%\end{Remark}

\begin{Remark}
  While filtered Floer or symplectic or Morse homology has been used
  for decades in symplectic dynamics, its interpretation as a
  persistence module was originally put forth in \cite{PS} to study
  certain dynamics questions for Hamiltonian diffeomorphisms in a
  context quite different from our setting.  The techniques were
  further developed in, e.g., \cite{UZ}, and the present paper is just
  one out of many applications of this machinery and the entire
  perspective. For instance, in \cite{Sh}, building on \cite{Se,ShZ},
  the persistence module interpretation of Floer homology was central
  to proving a variant of the Hofer--Zehnder conjecture for
  Hamiltonian diffeomorphisms of $\CP^n$.
\end{Remark}

The paper is organized as follows. We state our main results and
provide more context in Section \ref{sec:results}. In Section
\ref{sec:prelim}, we briefly cover the background material:
persistence modules, our conventions, symplectic homology and its
properties, and the $S^1$-equivariant spectral invariants over fields
of positive characteristic. Most of this material is well known, but
some of our results and definitions are not quite standard. Theorem
\ref{thm:spectral} is proved in Section
\ref{sec:proof-spectral}. Theorems \ref{thm:negative} and
\ref{thm:bounded} are established in Section \ref{sec:negative}.
Finally, in Section \ref{sec:pf-sh=1}, we prove Theorem
\ref{thm:sh=1}.

\subsection{Main results and a discussion}
\label{sec:results}

\subsubsection{Bounded equivariant spectral invariants in positive characteristic} 
\label{sec:spectral}
Our first result concerns an unexpected feature of $S^1$-equivariant
spectral invariants (aka capacities), over a field of characteristic
$p>0$. The definition of these invariants is identical to their
standard counterpart over $\Q$. Referring the reader to Section
\ref{sec:spectral-def} for details, for a Liouville domain $W$ we
denote them by $\cfp_1(W), \cfp_2(W),\dotsc$\ .

\begin{TheoremX}%[Spectral Invariants]
  \label{thm:spectral}
  For every Liouville domain $W\subset \R^{2n}$ and every prime $p$,
  the sequence
  \begin{equation}
    \label{eq:spec-seq}
    \cfp_1(W)\leq \cfp_2(W)\leq \cfp_3(W)\leq \dotsb
\end{equation}
is bounded.
\end{TheoremX}

This fact is in stark contrast with the behavior of the spectral
invariants $\cf_i^{(0)}$ over zero characteristic, which grow
approximately linearly. As a consequence of the theorem, for $p>0$ the
sequence $\cf_i^{(p)}(W)$ stabilizes whenever the action spectrum is
discrete, e.g., when the flow is non-degenerate.

\begin{Example}[Ellipsoids]
  \label{ex:ellipsoids}
  Let $\p W$ be an irrational ellipsoid in $\R^{2n}$ and
  $x_1, x_2, x_3, \dotsc$ be the closed Reeb orbits on $\p W$, listed
  in the order of increasing action: $\CA(x_1)< \CA(x_2)<\dotsb$. As
  is well known, $\cf_i^{(0)}(W)=\CA(x_i)$ for all $i\in\N$. On the
  other hand, let $x_k$ be the first $p$-iterated orbit on this
  list. Then it is not hard to see from the proof of Theorem
  \ref{thm:spectral} that $\cfp_i(W)=\cf_i^{(0)}(W)=\CA(x_i)$ for
  $i<k$ and $\cfp_i(W)=\cf_k^{(0)}(W)=\CA(x_k)$ for all $i\geq k$.
\end{Example}  

Theorem \ref{thm:spectral} and Example \ref{ex:ellipsoids} lead to 
several questions, which we will briefly discuss in the rest of this
section.

The first one concerns explicit calculations. Spectral invariants and
symplectic capacities are very difficult to calculate and there are
few classes of domains where it has been done. One of them is the
class of convex/concave toric domains where the capacities
$\cf_i^{(0)}$ were calculated in \cite{GH}. Although the case of $p>0$
is more involved, an explicit calculation might still be feasible and
provide (negative) answers to some of the further questions below.

The second question, or rather group of questions, is to what extent
Example \ref{ex:ellipsoids} is representative of the general
situation. For instance, in general without any assumptions on the
flow, does the sequence \eqref{eq:spec-seq} stabilize?  Furthermore,
set $\cfp_\infty(W)=\sup_i\cfp_i(W)$. This is a spectral invariant aka
a symplectic capacity. Is $\cfp_\infty(W)$ the action of a
$p$-iterated orbit, or even the smallest such action value? (We do
not see why the latter should be the case.)  Does the first strict
inequality $\cfp_i(W)<\cf_i^{(0)}(W)$ occur right after the first
$p$-iterated orbit? A closely related question is of the behavior
(e.g., monotonicity) of $\cfp_i(W)$ as a function of $p$.

Finally, it is entirely possible that the bound from Theorem
\ref{thm:spectral} is a general feature of $S^1$-equivariant spectral
invariants in positive characteristic, not specific to Liouville
domains in $\R^{2n}$. The proof of the theorem suggests that this is
so. (However, it is not entirely clear how large the class of domains
for which $S^1$-equivariant spectral invariants are defined, finite
and have reasonable properties is; see Section
\ref{sec:spectral-def}.)

\subsubsection{Negative mean indices} 
\label{sec:negative}
The Reeb variant of the Hofer--Zehnder (HZ) conjecture asserts that,
at least in the non-degenerate case, the Reeb flow on the boundary of
a star-shaped domain in $\R^{2n}$ has either exactly $n$ or infinitely
many prime closed Reeb orbits. While the question is in general
wide-open, it has been a subject of active research and by now the
conjecture, which also has numerous refinements, is supported by a
long list of related or partial results; see, e.g., \cite{AM,
  CGG:Reeb-HZ, DL, DLLW1, DLLW2, DLW, GG:LS, GGM, GM, GK, LZ}. Recall
that a \emph{Reeb pseudo-rotation} is a Reeb flow with finitely many
prime closed orbits. An example is the Reeb flow on an irrational
ellipsoid, but there are also examples of Reeb pseudo-rotations on
$C^\infty$-small perturbations of such ellipsoids with non-trivial
(e.g., ergodic) dynamics; see \cite{Ka}. Thus the conjecture can be
rephrased as that a Reeb pseudo-rotation must have exactly $n$ prime
closed Reeb orbits.

Furthermore, recall that a flow on the boundary of a Liouville domain
in $\R^{2n}$ is said to be dynamically convex if $\mu_-(x)\geq n+1$
for all closed Reeb orbits $x$, where $\mu_-$ is the lower
semi-continuous extension of the Conley--Zehnder index $\mu$. In
particular, when the flow is non-degenerate, we have $\mu_-(x)=\mu(x)$
and the requirement is simply that $\mu(x)\geq n+1$; see \cite{HWZ}
where the notion of dynamical convexity was introduced. Ordinary
geometric convexity implies dynamical convexity as shown in that
paper, but the converse is not true; see \cite{CE1, CE2} for the first
counterexamples.

For a non-degenerate dynamically convex flow on the boundary of
symmetric star-shaped domain, the HZ-conjecture conjecture was proved
in \cite{CGG:Reeb-HZ} and we refer the reader to that paper for a
detailed discussion of the conjecture, its variants and refinements,
and past results. As a closely related question, we conjecture that a
non-degenerate Reeb pseudo-rotations is necessarily dynamically
convex. Let us refer to this statement as the dynamical convexity
conjecture.

Dynamical convexity has numerous simple linear algebra consequences
about the behavior of the sequence of the indices $\mu_-(x^k)$ for the
iterates $x^k$. For instance, this sequence is strictly increasing and
$\hmu(x)>2$, where $\hmu$ is the mean index of $x$; see, e.g.,
\cite{GG:LS} for a discussion and further references. In particular,
$\hmu(x)\geq 0$. Here, as a step towards the proof of the dynamical
convexity conjecture, we show that the latter condition is
automatically satisfied for non-degenerate Reeb pseudo-rotations.

\begin{TheoremX}
  \label{thm:negative}
  Assume that the Reeb flow on the boundary of a star-shaped domain
  $W\subset \R^{2n}$ is non-degenerate and has a closed orbit with
  negative mean index. Then the flow has infinitely many prime orbits.
\end{TheoremX}

The proof of this theorem given in Section \ref{sec:negative} heavily
relies on the machinery developed in \cite{CGG:Reeb-HZ} and the
analysis of the symplectic homology persistence module. 

\subsubsection{From symplectic homology to periodic orbits} 
\label{sec:sh=1}
Let $\SH^t(W;\F)$ stand for the filtered symplectic homology of a
star-shaped domain $W \subset \R^{2n}$ with coefficients in a field
$\F$.  The key technical result from \cite[Thm.\ C]{CGG:Reeb-HZ} is
that $\dim\SH^t(W;\F)=1$ for all $t>0$ and every ground field $\F$,
whenever the flow on the boundary $\p W$ is a dynamically convex
pseudo-rotation. This fact is crucial to other proofs in
\cite{CGG:Reeb-HZ}, although it is not clear what exactly the formal
consequences of this property are. For instance, we conjecture that in
the non-degenerate case this condition is equivalent to the flow being
a dynamically convex pseudo-rotation. (One of the issues arising for
degenerate flows is the Floer-theoretic visibility of closed orbits:
when the local symplectic homology of all iterates of $x$ vanishes,
the filtered symplectic homology carries no information about $x$.)
Below we explore the implications of this condition including the
behavior of the local homology of closed Reeb orbits. Our first result
here, setting the stage for further development, is the following.
\begin{TheoremX}
  \label{thm:bounded}
  Assume that the Reeb flow on the boundary of a Liouville domain
  $W\subset \R^{2n}$, i.e., an image of an exact symplectic embedding
  of a Liouville domain, is a Reeb pseudo-rotation. Then
  $\dim \SH^t(W)$ is bounded as a function of $t\in\R$ for any ground
  field $\F$.
\end{TheoremX}

The proof of this result is implicitly contained in
\cite{CGG:Reeb-HZ}. We will outline the proof in Section
\ref{sec:bounded-pf}. When, in addition, $W$ is star-shaped and the
flow is dynamically convex, $\dim \SH^t(W)=1$ for all $t>0$;
\cite[Thm.\ C]{CGG:Reeb-HZ}.  One might expect the converse of Theorem
\ref{thm:bounded} to also hold, but this is a very subtle and possibly
deep question. At the time of writing, it is not even known whether
the flow must be a pseudo-rotation when $\dim \SH^t(W)=1$ for all
$t>0$ and the flow is non-degenerate; cf.\ Theorem \ref{thm:sh=1}
below. We believe that these two conditions combined are equivalent to
that the flow is a pseudo-rotation.

\begin{Remark}
  Another situation when $\dim \SH^t(W;\Q)=1$ for all $t>0$ is when
  the flow on the boundary of a star-shaped domain is Besse; see
  \cite{GRT}. One can think of such a flow as an example of a
  generalized pseudo-rotation and in this case the flow is again
  dynamically convex.
\end{Remark}  

Recall that an isolated closed Reeb orbit $x$ is said to be
$\F$-visible if $\SH(x;\F)\neq 0$, where $\SH(x;\F)$ is the local
symplectic homology of $x$ over a field $\F$. Likewise, we could have
said that $x$ is equivariantly $\F$-visible if $\COH(x;\F)\neq 0$,
where $\COH(x;\F)$ is the $S^1$-equivariant local symplectic homology
of $x$ over $\F$.  However, a closed Reeb orbit is $\F$-visible if and
only if it is equivariantly $\F$-visible. This follows from the Gysin
long exact sequence for local symplectic homology; see Remark
\ref{rmk:visibility}. %\eqref{eq:loc-Gysin}.
Hence there is no need to distinguish between these two types of
visibility.

Furthermore, recall that $x$ is said to be strongly non-degenerate if
all iterates $x^k$ are non-degenerate. Then either all indices
$\mu\big(x^k\big)$ have the same parity or the parity alternates. We
call $x$ alternating in the latter case and non-alternating in the
former.  When $x$ is non-degenerate, it is always $\F$-visible when
$\charr \F=2$.  For $\charr \F\neq 2$ such an orbit $x$ is
$\F$-visible if and only if it is good, i.e., $x$ is not an even
iterate of an alternating orbit. Finally, let $\CA(x)$ be the action
(aka the period) of $x$.

\begin{TheoremX}
  \label{thm:sh=1}
  Let $W$ be a Liouville domain in $\R^{2n}$ such that all closed Reeb
  orbits on the boundary $\p W$ are isolated.
  \begin{itemize}

  \item[\reflb{sh1}{\rm{(i)}}] Assume that $\dim \SH^t(W;\Q)=1$ and
    $\dim \SH^t(W;\F)=1$ for some field $\F$ of possibly positive
    characteristic and all $t>0$. Then $\SH(x;\F)$ is supported in two
    consecutive degrees for every $\F$-visible closed Reeb orbit
    $x$. Furthermore, let the support of $\SH(x;\F)$ be $\{q,q+1\}$.
    Then $q$ is the lowest degree where $\COH(x;\F)\neq 0$.

  \item[\reflb{sh2}{\rm{(ii)}}] Assume that the Reeb flow on $\p W$ is
    non-degenerate and $\dim \SH^t(W;\F)=1$ for $\F=\Q$ and $\F=\F_2$
    and all $t>0$. Then the Reeb flow on $\p W$ is dynamically convex,
    has at most $n$ non-alternating prime orbits, and all (good and
    bad) closed Reeb orbits $x$ have the same ratio
    $\CA(x)/\hmu(x)$.
    
  \end{itemize}  
\end{TheoremX}

Here Part \ref{sh1} shows that closed orbits of the flow behave
essentially as if they were non-degenerate.  (There are no known
examples of Reeb flows for which the conditions of Theorem
\ref{thm:sh=1} are satisfied, but the flow is degenerate.) The
assertion that all closed Reeb orbits $x$ have the same ratio
$\CA(x)/\hmu(x)$ is crucial for the Hofer--Zehnder conjecture type
results; see, e.g., \cite{CGG:Reeb-HZ,GG:LS}.  Note, however, that in
Part \ref{sh2}, we do not claim that the flow is a pseudo-rotation,
i.e., it has finitely many closed Reeb orbits. Nor do we assert that
the number of prime good orbits is exactly $n$. (The latter statement
would follow from the former due to \cite{LZ,GK}.)

\begin{Remark}
  \label{rmk:sh=1}
  As is easy to see, $\dim \SH(x;\F)$ is either 2 or 0 for any field
  $\F$ whenever $\dim \SH^t(W;\F)=1$ for all $t$ near $\CA(x)$; cf.\
  \cite{CGG:Reeb-HZ}. We will recall a proof of this fact in Section
  \ref{sec:support} while proving Part \ref{sh1} of the theorem.
\end{Remark}  

\section{Preliminaries}
\label{sec:prelim}

\subsection{Persistence modules}
\label{sec:persistence}
In this section, we review the definition of persistence modules,
focusing on the class of modules relevant to our purposes, and briefly
touch upon their properties. We refer the reader to \cite{PRSZ} for a
general introduction to persistence modules, their applications to
geometry and analysis and further references, although the type of
persistence modules they consider is somewhat more narrow than the
ones we work with here, and also to \cite{BV, CB} for some of the more
general results.

Fix a field $\F$, which we will suppress in the notation. A
\emph{persistence module} $(\CV,\pi)$ is a family of vector spaces
$\CV_s$ over $\F$ parametrized by $s\in \R$, together with a
functorial family $\pi$ of structure maps. These are linear maps
$\pi_{st}\colon \CV_s\to \CV_t$, where $s\leq t$, and functoriality is
understood as that $\pi_{sr}=\pi_{tr}\pi_{st}$ whenever
$s\leq t\leq r$ and $\pi_{ss}=\id$. In what follows, we often suppress
$\pi$ in the notation and simply refer to $(\CV,\pi)$ as $\CV$.  In
such generality, the concept is not particularly useful, and one
usually imposes additional conditions on the vector spaces $\CV_t$ and
the structure maps $\pi_{st}$. These conditions vary depending on the
context. Below we spell out the framework most suitable for our
purposes.

Namely, we require that there is a closed, bounded from below, nowhere
dense closed subset $\CS\subset \R$, which is called the
\emph{spectrum} of $\CV$, and the following four conditions are met:
\begin{itemize}

\item[\reflb{PM1}{\rm{(PM1)}}] The persistence module $\CV$ is
  \emph{locally constant} outside $\CS$, i.e., $\pi_{st}$ is an
  isomorphism whenever $s\leq t$ are in the same connected component
  of $\R\setminus \CS$.

\item[\reflb{PM2}{\rm{(PM2)}}] The persistence module $\CV$ is
  \emph{$q$-tame}: $\pi_{st}$ has finite rank for all $s<t$.
 
\item[\reflb{PM3}{\rm{(PM3)}}] \emph{Left-semicontinuity}: For all
  $t\in\R$,
  \begin{equation}
    \label{eq:semi-cont}
    \CV_t=\varinjlim_{s<t} \CV_s.
  \end{equation}

\item[\reflb{PM4}{\rm{(PM4)}}] \emph{Lower bound}: $\CV_s=0$ when
  $s<s_0$ for some $s_0\in\R$. (Throughout the paper we will assume
  that $s_0=0$.)

\end{itemize}

A few comments on this definition are in order. First, note that as a
consequence of \ref{PM1} and \ref{PM2}, $\CV_s$ is finite-dimensional
and \ref{PM3} is automatically satisfied when $s\not\in
\CS$. Furthermore, in several instances of interest to us, $\CV_s$ is
naturally defined only for $s\not\in\CS$, and then the definition is
extended to all $s\in\R$ by \eqref{eq:semi-cont}. By \ref{PM4}, we can
always assume that $s_0\leq \inf \CS$, i.e., $\CS$ is bounded from
below. We emphasize, however, that $\CS$ is not assumed to be bounded
from above and it is actually not in many examples we are interested
in. In what follows, it will be sometimes convenient to include
$s=\infty$ by setting
$$
\CV_\infty=\varinjlim_{s\to\infty} \CV_s.
$$
Finally, in all examples we encounter here $\CS$ has zero measure and,
in fact, zero Hausdorff dimension, but this fact is never used in the
paper.  Graded persistence modules are defined in a similar fashion.

A basic example motivating requirements \ref{PM1}--\ref{PM4} is that
of the sublevel homology of a smooth function.

\begin{Example}[Homology of sublevels] Let $M$ be a smooth manifold
  and $f\colon M\to \R$ be a proper smooth function bounded from
  below. Set $\CV_s:=\H_*\big(\{f<s\};\F\big)$ with the structure maps
  induced by inclusions. No other requirements are imposed on $f$,
  e.g., $f$ need not be Morse. It is not hard to see that conditions
  \ref{PM1}--\ref{PM4} are met with $\CS$ being the set of critical
  values of $f$. We note that one can have $\dim \CV_s=\infty$ for
  $s\in\CS$ already when $M=S^1$, unless $f$ meets some additional
  conditions on $f$, e.g., that $f$ is real analytic or the critical
  points of $f$ are isolated.
\end{Example}

An \emph{interval persistence module} $\F_{(a,\,b]}$, where
$-\infty<a<b\leq \infty$, is defined by 
$$
V_s:=\begin{cases}
  \F & \text{ when } s\in (a,\,b],\\
  0 & \text{ when } s \not\in (a,\,b],
\end{cases}
$$
and $\pi_{st}=\id$ if $a<s\leq t\leq b$ and $\pi_{st}=0$
otherwise. Interval modules are examples of simple persistence
modules, i.e., persistence modules that cannot be decomposed as a
(non-trivial) direct sum of other persistence modules.

A key fact that we will use in the paper is the normal form or
structure theorem asserting that every persistence module meeting the
above conditions can be decomposed as a direct sum of a countable
collection (i.e., a countable multiset) of interval persistence
modules. Moreover, this decomposition is unique up to re-ordering 
the sum. (In fact, conditions \ref{PM1}--\ref{PM4} are far from
optimal and can be considerably relaxed.) We refer the reader
\cite{BV, CB} for proofs of this theorem for the class of persistence
modules considered here and further references, and also, e.g., to
\cite{CZCG, ZC} for previous or related results.

This multiset $\CB(\CV)$ of intervals entering this decomposition is
referred to as the \emph{barcode} of $\CV$ and the intervals occurring
in $\CB(\CV)$ as \emph{bars}. When $\CV$ is graded, we have a degree
assigned to every bar.

\subsection{Filtered symplectic homology}
\label{sec:SH0}
In this section, we briefly discuss filtered symplectic homology as a
persistence module following \cite{CGGM:Entropy, CGGM:Inv}.

\subsubsection{Conventions, notation, semi-admissible Hamiltonians, the Floer
equation}
\label{sec:setting}
We start by spelling out our conventions and notation, which are
essentially identical to the ones used in \cite{CGGM:Entropy,
  CGGM:Inv, GG:LS}, and also recalling several elementary properties
of (semi-)admissible Hamiltonians to be used later.

Let $\alpha$ be the contact form on the boundary $M=\p W$ of a
Liouville domain $W^{2n\geq 4}$. We will also use the same notation
$\alpha$ for a primitive of the symplectic form $\omega$ on
$W$. Furthermore, we will require that $c_1(TW)=0$.  As usual, denote
by $\WW$ the symplectic completion of $W$, i.e.,
$$
\WW=W\cup_M M\times [1,\,\infty)
$$
with the symplectic form $\omega=d\alpha$ extended from $W$ to
$M\times [1,\infty)$ as
$$
\omega := d(r\alpha),
$$
where $r$ is the coordinate on $[1,\,\infty)$. Sometimes it is
convenient to have the function $r$ also defined on a collar of
$M=\p W$ in $W$. Thus we can think of $\WW$ as the union of $W$ and
$M\times [1-\eta,\,\infty)$ for small $\eta>0$ with
$M\times [1-\eta,\, 1]$ lying in $W$ and the symplectic form given by
the same formula.

The action spectrum $\CS(\alpha)\subset (0,\infty)$ is the set of all
actions (aka periods) of closed orbits of the Reeb flow of $\alpha$
on $M$. This is a closed, zero measure set.

We say that a closed Reeb orbit $y$ in $M$ is the $k$th \emph{iterate}
of a closed orbit $x$ and write $y=x^k$ if $y(t)=x(kt)$ up to the
choice of initial conditions on $x$ and $y$. (In what follows we will
always identify the closed orbits $t\mapsto x(t)$ and
$t\mapsto x(t+\theta)$ for all $\theta\in\R$.) A closed orbit $x$ is
\emph{prime} if it is not an iterate of any other closed orbit.

A closed Reeb orbit $x$ is said to be \emph{non-degenerate} if $1$ is
not an eigenvalue of its Poincar\'e return map. The Reeb flow on $M$
is non-degenerate if all closed orbits, prime and iterated, are
non-degenerate. We denote the Conley--Zehnder index of $x$ by
$\mu(x)$; see, e.g., \cite{SaZ}.  When $x$ is non-degenerate, we call
it \emph{non-alternating} if the number, counted with algebraic
multiplicity, of real negative eigenvalues in $(-1,0)$ of its
Poincar\'e return map is even. Otherwise, we call $x$
\emph{alternating}. A prime orbit $x$ is non-alternating if and only
if all entries in the sequence $\mu\big( x^k \big)$ have the same
parity and alternating if and only if the parity alternates. A
non-degenerate even iterate of an alternating closed orbit is called
\emph{bad}; all other non-degenerate closed orbits are called
\emph{good}.

Most of the Hamiltonians $H\colon \WW\to \R$ considered in this
section are constant on $W$ and depend only on $r$ outside $W$, i.e.,
$H=h(r)$ on $M\times [1,\,\infty)$, where the $C^\infty$-smooth
function $h\colon [1,\,\infty)\to \R$ is required to meet the
following three conditions:
\begin{itemize}
\item $h$ is strictly monotone increasing;
\item $h$ is convex, i.e., $h''\geq 0$, and $h''>0$ on $(1,\, \rmax)$
  for some $\rmax>1$ depending on~$h$;
\item $h(r)$ is linear, i.e., $h(r)=ar-c$, when $r\geq \rmax$.
\end{itemize}
In other words, the function $h$ changes from constant on $W$ to
convex in $r$ on $M\times [1,\, \rmax]$, strictly convex on the
interior, to linear in $r$ on $M\times [\rmax,\, \infty)$.

We will refer to $a$ as the \emph{slope} of $H$ (or $h$) and write
$a=\slope(H)$. The slope is often, but not always, assumed to be
outside the action spectrum of $\alpha$, i.e.,
$a\not\in\CS(\alpha)$. We call $H$ \emph{admissible} if
$H|_W=\const<0$ and \emph{semi-admissible} when $H|_W\equiv 0$. (This
terminology differs from the standard usage, and we emphasize that
\emph{admissible Hamiltonians are not semi-admissible}.) When $H$
satisfies only the last of the three conditions, we call it
\emph{linear at infinity}.

The difference between admissible and semi-admissible Hamiltonians is
just an additive constant: $H- H|_W$ is semi-admissible when $H$ is
admissible. Hence the two Hamiltonians have the same filtered Floer
homology up to an action shift. For our purposes, semi-admissible
Hamiltonians are notably more suitable due to the $H|_W\equiv 0$
normalization.

The Hamiltonian vector field $X_H$ is determined by the condition
$$
\omega(X_H,\, \cdot)=-dH,
$$
and, on $M\times [1,\,\infty)$, 
$$
X_H=h'(r) R_\alpha,
$$
where $R_\alpha$ is the Reeb vector field. We denote the Hamiltonian
flow of $H$ by $\varphi_H^t$, the Reeb flow of $\alpha$ by
$\varphi_\alpha^t$, where $t\in \R$, and the Hamiltonian
diffeomorphism generated by $H$ by $\varphi_H:=\varphi_H^1$.

By a \emph{$\tau$-periodic orbit} of $H$ we will mean one of several
closely related but distinct objects. It can be a $\tau$-periodic
orbit of $\varphi_H$ and then $\tau\in\N$. Alternatively, it can stand
for a $\tau$-periodic orbit of the flow $\varphi_H^t$ with
$\tau\in (0,\,\infty)$. Furthermore, working with periodic orbits of
flows or maps, we might or might not have the initial condition
fixed. For instance, without an initial condition fixed, a
non-constant 1-periodic orbit of the flow of $H$ gives rise to a whole
circle of 1-periodic orbits (aka fixed points) of
$\varphi_H$. Likewise, a prime $\tau$-periodic orbit of $\varphi_H$
comprises a set of $\tau$-periodic points.  In most cases, the exact
meaning should be clear from the context and is often immaterial; when
the difference is essential we will specify whether an orbit is of the
flow or the diffeomorphism and if the initial condition is fixed or
not.

Every $T$-periodic orbit $z$ of the Reeb flow with $T<a=\slope(H)$
gives rise to a 1-periodic orbit $\tz=(z,r_*)$ of the flow of $H$ with
$r_*$ determined by the condition
\begin{equation}
  \label{eq:level}
h'(r_*)=T.
\end{equation}
Clearly, $\tz$ lies in the shell $1<r<\rmax$, and we have a one-to-one
correspondence between 1-periodic orbits of $H$ and the periodic
orbits of $\varphi_\alpha^t$ with period $T<a$ whenever
$a\not\in \CS(\alpha)$. In the pair $\tz=(z,r_*)$, we usually view
$\tz$ as a 1-periodic orbit of the flow $\varphi^t_H$ of $H$ or a
circle of fixed points of the Hamiltonian diffeomorphism $\varphi_H$,
while $z$, contrary to what the notation might suggest, is
parametrized by the Reeb flow but not as a projection of $\tz$ to
$M$. (By \eqref{eq:level}, the two parametrizations of $z$ differ by
the factor of $h'(r_*)=T$.) Fixing an initial condition on $z$
determines an initial condition on $\tz$, and the other way around. In
particular, $z$ gives rise to a whole circle $\tz(S^1)$ of fixed
points of $\varphi_H$.

We say that a $T$-periodic orbit $z$ of the Reeb flow is
\emph{isolated} (as a periodic orbit) if for every $T'> T$ it is
isolated among periodic orbits with period less than $T'$. Clearly,
all periodic orbits of $\alpha$ are isolated if and only if for every
$T'$ the number of periodic orbits with period less than $T'$ is
finite.  For instance, a non-degenerate periodic orbit is isolated.
Note that $\tz$ is isolated as a 1-periodic orbit of the flow of $H$
if $z$ is isolated. No fixed point of $\varphi_H$ on $\tz(S^1)$ is
isolated, but $\tz$ is Morse--Bott non-degenerate, as the set of fixed
points $\tz(S^1)$, if and only if $z$ is non-degenerate.

The action functional $\CA_H$ is defined by
\[
  \CA_H(\gamma)=\int_\gamma\hat{\alpha}-\int_{S^1} H(\gamma(t))\, dt,
\]
where $\gamma\colon S^1=\R/\Z\to \WW$ is a smooth loop in $\WW$ and
$\hat{\alpha}$ is the Liouville primitive $\alpha$ of $\omega$ on $W$
and $\hat{\alpha}=r\alpha$ on $M\times [1-\eta,\,\infty)$ for a
sufficiently small $\eta>0$. More explicitly, when
$\gamma\colon S^1\to M\times [1,\,\infty)$, we have
\[
\CA_H(\gamma)= \int_{S^1} r(\gamma(t))\alpha\big(\gamma'(t)\big)\, dt
- \int_{S^1} h\big(r(\gamma(t))\big)\, dt.
\]
Thus when $\gamma=\tz=(z,r_*)$ is a 1-periodic orbit of $H$, the
action can be expressed as a function of $r_*$ only:
$$
\CA_H(\tz)=A_h(r_*),
$$
where
\begin{equation}
  \label{eq:AH}
  A_h\colon [1,\,\infty)\to [0,\,\infty)\textrm{ is given by } A_h(r)=r
  h'(r)-h(r).
\end{equation}
Sometimes we will also denote this \emph{action function} by $A_H$.
It is easy to see that this is a monotone increasing function;
\cite[Sec.\ 2.2]{CGGM:Entropy}.

Next fix an almost complex structure $J$ on $\WW$ satisfying the
following conditions:
\begin{itemize}
\item $J$ is compatible with $\omega$, i.e., $\omega(\cdot,\, J\cdot)$
  is a Riemannian metric,
\item $J r\p /\p r=R_\alpha$ on the cylinder $M\times [1,\,\infty)$,
\item $J$ preserves $\ker (\alpha) $.  
\end{itemize}
The last two conditions are equivalent to that
$$
%\begin{equation}
%  \label{eq:complex_strc}
dr\circ J=-r\alpha.
% \end{equation}
$$
We call such almost complex structures \emph{admissible}. If the first
condition still holds on $\WW$, and the second and the third
conditions are met only outside a compact set while within a compact
set $J$ can be time-dependent and 1-periodic in time, we call $J$
\emph{admissible at infinity}.

Let $H$ be a Hamiltonian linear at infinity and let $J$ be an
admissible at infinity almost complex structure. The Floer equation is
the $L^2$-anti-gradient flow equation for $\CA_H$.  With our
conventions, this equation reads
\begin{equation}
\label{eq:floer_2}
\p_s u-J\big(\p_t u- X_H(u)\big)=0,
\end{equation}
where $u\colon \R\times S^1\to \WW$ and $(s,t)$ are the coordinates on
$\R\times S^1$ with $S^1=\R/\Z$. By construction, the function
$s\mapsto \CA_H\big(u(s,\,\cdot)\big)$ is decreasing.

Note that the leading term of this equation is the $\p$-operator, but
not the $\bar{\p}$-operator.  In other words, when $H\equiv 0$,
solutions of \eqref{eq:floer_2} are anti-holomorphic rather than
holomorphic curves. Nonetheless, the standard properties of the
solutions of the Floer equation readily translate to our setting,
e.g., via the change of variables $s\mapsto -s$. We will often refer
to solutions $u$ of the Floer equation as \emph{Floer cylinders}.

By construction, when $u$ is asymptotic to $\ty_1$ at $+\infty$ and to
$\ty_0$ at $-\infty$, we have
$$
% \begin{equation}
% \label{eq:Energy-Action}
E(u)=\CA_H(\ty_0)-\CA_H(\ty_1),  %=A_H(r^+)-A_H(r^-).
% \end{equation}
$$
where the energy $E(u)$ is given by
$$
%\begin{equation}
%\label{eq:energy}  
E(u)=\int_{S^1\times\R}\|\p_s u\|^2\,dt\,ds.
%\end{equation}
$$
When $H$ is semi-admissible, we can be more precise.  Namely, let $u$
be asymptotic to $\tx=(x, r^+)$ at $-\infty$ and to $\ty=(y,r^-)$ at
$+\infty$. Then
$$
%\begin{equation}
%  \label{eq:Energy-Action2}
E(u)=\CA_H(\tx)-\CA_H(\ty)=A_H(r^+)-A_H(r^-).
%\end{equation}
$$
Here, due to our conventions, $r^+\geq r^-$, and hence the notation.

\subsubsection{Symplectic homology}
\label{sec:SH-def}
The first versions of symplectic homology were introduced in
\cite{CFH, Vi} and since then the theory has been further
developed. We refer the reader to, e.g., \cite{Abu, Bo, BO0, BO,
  BO:Gysin, Se:biased} for the definitions and basic constructions,
which we mainly omit. Here, treating equivariant and non-equivariant
symplectic homology as persistence modules, we closely follow
\cite{CGGM:Entropy, GG:LS}.

Fix a field $\F$ which we suppress in the notation.  Let $H$ be linear
at infinity with the slope outside $\CS(\alpha)$ and $J$ admissible at
infinity. Then the filtered Floer homology $\HF^\tau(H)$ of $H$ over
$\F$ is defined regardless of whether $H$ is non-degenerate or not.
Indeed, let $\CS(H)$ be the action spectrum of $H$. When
$\tau\not\in \CS(H)$, we set $\HF^\tau(H):=\HF^\tau(\tH)$, where $\tH$
is a sufficiently small perturbation of $H$, depending on $\tau$. For
$\tau\in\CS(H)$, we define $\HF^\tau(H)$ as the left limit of
$\HF^{\tau'}(H)$ as $\tau'\to \tau-$ and $\tau'\not\in\CS(H)$; cf.\
\ref{PM3}. So defined, the filtered Floer homology is a graded
persistence module over $\F$.

For instance, assume that the Reeb flow is non-degenerate and $H$ is
semi-admissible with $a=\slope(H)\not\in\CS(\alpha)$. Then, as we have
mentioned in the previous section, every non-trivial 1-periodic orbit
of $H$ has the form $\tz=(z,r_*)$, where $z$ is a closed Reeb orbit of
period less than $a$. Then one can find a non-degenerate perturbation
$\tH$ such that $z$ gives rise to two generators $\hz$ and $\cz$ in
$\CF(\tH)$. If the grading is defined and $m$ is the Conley--Zehnder
index of $z$, the degree of $\hz$ is $m+1$ and the degree of $\cz$ is
$m$. In addition, $\tH$ is a $C^2$-small function on $W$ with each
critical point contributing one generator to the Floer complex.

In contrast with Hamiltonian Floer homology on closed manifolds, here
a homotopy does not in general induce a continuation map between the
homology for two distinct linear at infinity Hamiltonians. Let $H_s$,
$s\in\R$, be a homotopy between two such Hamiltonians $H_0$ and $H_1$,
i.e., $H_s$ is a family of linear at infinity Hamiltonians such that
$H_s=H_0$ when $s$ is close to $-\infty$ and $H_s=H_1$ when $s$ is
close to $+\infty$. (In what follows, we will take the liberty to have
homotopies parametrized by $s\in [0,\,1]$ or some other finite
interval rather than $\R$.) There are two cases where a homotopy gives
rise to a map in Floer homology.

The first one is when all Hamiltonians $H_s$ have the same slope. Then
the homotopy induces a continuation map
\[
\HF^\tau(H_0)\to \HF^{\tau +C}(H_1),
\]
shifting the action filtration by
\[
C=\int_{-\infty}^\infty
\, \max_{z\in\WW} \, \max\{0,-\p_s H_s (z)\}\,ds.
\]
Moreover, it is well known and not hard to show that $\HF(H)$ does not
change as long as $\slope(H)$ stays outside of $\CS(\alpha)$.

The second case is when $H_s$ is monotone increasing, i.e., the
function $s\mapsto H_s(z)$ is monotone increasing for all $z\in
\WW$. In particular, the function $s\mapsto \slope(H_s)$ is also
monotone increasing. Note that, while $\slope(H_0)$ and $\slope(H_1)$
are still required to be outside $\CS(\alpha)$, the intermediate
slopes $\slope(H_s)$ can pass through the points of
$\CS(\alpha)$. Such a homotopy induces a map
$$
\HF^\tau(H_0)\to \HF^{\tau}(H_1),
$$
preserving the action filtration. Moreover, it is well known and not
hard to show that the global homology $\HF(H_s)$ does not change as
long as $\slope(H_s)$ stays outside of $\CS(\alpha)$.

The \emph{filtered symplectic homology} $\SH^\tau(\alpha)$ or
$\SH^\tau(W)$ is defined as
\begin{equation}
  \label{eq:SH}
\SH^\tau(W):=\varinjlim_H \HF^\tau(H),
\end{equation}
where traditionally the limit is taken over all Hamiltonians linear at
infinity and such that $H|_W<0$. Since admissible (but not
semi-admissible) Hamiltonians form a co-final family, we can limit $H$
to this class.  When working with this definition, it is useful to
keep in mind that, as is not hard to see,
\begin{equation}
  \label{eq:spectra-conv}
\CS(H)\to\{ 0\}\cup\CS(\alpha)
\end{equation}
uniformly on compact intervals. Furthermore, even though $H|_W=0$ for
semi-admissible Hamiltonians rather than $H|_W<0$, one can take the
limit over all such Hamiltonians in \eqref{eq:SH}. This fact, which
readily follows from the definition, is useful for computations, and
this is how we will usually treat \eqref{eq:SH} in what follows.

What is even more useful for computations is that no limit is needed
in the definition of the filtered symplectic homology. Namely, we have
$\HF(H)=\SH^a(W)$ for every semi-admissible Hamiltonian $H$ with
$a=\slope(H)\not\in \CS(\alpha)$. Moreover, for any two such
Hamiltonians $H_0$ and $H_1$ the persistence modules $\HF^\tau(H_0)$
and $\HF^\tau(H_1)$, differ in the obvious sense by a
reparametrization. This reparametrization can be made arbitrary close
to the identity for a suitable choice of $H_0$ and $H_1$. Furthermore,
let us truncate the persistence module $\SH^\tau(W)$ at $\tau=a$ by
turning all bars passing through $a$ into infinite bars and removing
all bars beginning above $a$. (This is not the standard truncation of
persistence modules.) Then the persistence module $\HF(H)$ is again a
reparametrization of this truncated symplectic homology, and moreover
the reparametrization is arbitrarily close to the identity for a
suitable choice of $H$. We refer the reader to \cite{CGGM:Entropy} and
\cite{CGGM:Inv} for the proofs of these facts.

The family $\SH^\tau(W)$ with natural structure maps is a graded
persistence module.

For an interval $I=[a,b)$ with end-points outside $\CS(\alpha)$, the
symplectic homology $\SH^I(W)$ is defined as the limit of the Floer
homology $\HF^I(H)$ similarly to \eqref{eq:SH}. This definition
extends to all intervals ``by continuity''. Thus, for instance,
$\SH^\tau(W)=\SH^{(-\infty,\tau)}(W)$.  We have the long exact
sequence
$$
\cdots\to \SH^a(W)\to \SH^b(W)\to \SH^I(W)\to \cdots .
$$
Furthermore, when $\delta>0$ is sufficiently small,
e.g., $\eps<\min\CS(\alpha)$,
\begin{equation}
\label{eq:SH-small-delta}
\SH^\eps(W):=\H_*(W,\p W)[n];
\end{equation}
see, e.g., \cite{BO, Vi}. 

Replacing the Floer homology by the $S^1$-equivariant Floer homology
everywhere in the above constructions, we obtain the filtered
$S^1$-equivariant symplectic homology, which we denote by
$\COH^\tau(W)$ or $\COH^I(W)$, etc. The family $\COH^\tau(W)$ with
natural structure maps is again a graded persistence module.  Just as
the ordinary symplectic or Floer homology, the equivariant symplectic
homology is actually defined over $\Z$ and hence over any ring. As a
consequence, the universal coefficient formula applies.  We set
$\COH^+(W):=\COH^{(\eps,\infty)}(W)$ and $\COH^-(W):=\COH^{\eps}(W)$,
where $\eps>0$ is sufficiently small, i.e., $0<\eps<\min\CS(\alpha)$.

\begin{Example}
  \label{ex:CH}
When $W\subset \R^{2n}$ is a star-shaped domain,
\begin{equation}
  \label{eq:CH}
  \dim\COH_m^+(W)=
  \begin{cases}
    1\textrm{ when } m=n+1+2i,\, i\geq 0,\\
    0\textrm{ otherwise;}
  \end{cases}
\end{equation}
see, e.g., \cite{BO:Gysin}.
\end{Example}

Clearly, we have the long exact sequence
\begin{equation}
  \label{eq:COH-long}
  \cdots\to \COH^-_m(W)\to \COH_m(W)\to \COH^+_m(W)
  \stackrel{\delta}{\longrightarrow}
  \COH_{m-1}^-(W)\to\cdots .
\end{equation} 
Furthermore, up to a shift of degrees,
\begin{equation}
  \label{eq:COH-}
 \COH^-(W)=\H_{*}(W,\p W)\otimes \H_*(BS^1)=\H_{*}(W,\p W)[u];
\end{equation}  
cf.\ \eqref{eq:SH-small-delta}.

When the Reeb flow on $\p W$ is non-degenerate and $\F=\Q$ or more
generally $\charr\F=0$, we can interpret $\COH^+(W;\Q)$ as the
homology of a certain complex generated by good closed Reeb orbits;
see \cite{GG:LS}. (For $\COH(W)$, one has to add a sequence of
generators for each critical point in $W$ of an admissible
non-degenerate Hamiltonian.) However, this is no longer true when $\F$
has positive characteristic. This fact is crucial for
Theorem~\ref{thm:spectral}.

Filtered symplectic homology is functorial with respect to exact
symplectic embeddings of Liouville domains. Namely, such an embedding
$V \subset W$ gives rise to morphisms of persistence modules
$\SH^\tau(W)\to \SH^\tau(V)$ and $\COH^\tau(W)\to \COH^\tau(V)$ called
the \emph{Viterbo transfer}; \cite{Vi} and, e.g., \cite{GH}.

Filtered equivariant and non-equivariant symplectic homology fit
together into the Gysin exact sequence
$$
\cdots \to \SH_m^I(W)\to \COH^I_m(W)
\stackrel{\DD}{\longrightarrow}  \COH^I_{m-2}(W)\to
\SH^I_{m-1}(W)\to 
\cdots 
$$
over any ring; see, e.g., \cite{BO:Gysin}. Strictly speaking, the
construction of symplectic homology in \cite{BO:Gysin} uses $\Q$ as
the coefficient field and a semi-infinite interval as $I$, but it
carries over word-for-word to $\Z$, and hence to any ring, and an
arbitrary interval $I$; see, e.g., \cite[Sect.\ 2]{GG:LS} for a
detailed discussion of the shift operator $\DD$ and also \cite{GH},
and \cite{Abu} for a non-equivariant construction over $\Z$.

Finally, there is a variant of the Leray spectral sequence 
\begin{equation}
  \label{eq:Leray}
E^r \Rightarrow \COH^I(W)
\end{equation}
with
$$
E^1=E^2=\SH^I(W)\otimes \H_*(BS^1)
= \SH^I(W)[u], \textrm{ where } \deg u = 2 .
$$
For instance, $\COH^I(W)=0$ whenever $\SH^I(W)=0$. In this spectral
sequence, $E^r_{q,p}=0$ when $q$ is odd since $\deg u=2$. Hence, all
odd differentials vanish, and $E^1=E^2$, $E^3=E^4$, etc. Note also
that the notation $E^2=\SH^I(W)[u]$, where we think of $E^2$ as
polynomials in $u$ with coefficients in the symplectic homology is
somewhat misleading. The pages are not naturally modules over
$\F[u]=\H_*(BS^1;\F)$ and the differentials do not commute with the
multiplication by $u$. However, the pages and $\COH^I(W)$ are modules
over $\F[u^{-1}]=\H^*(BS^1;\F)$. On $E^\infty$, the multiplication by
$u^{-1}$ is induced by the shift operator $\DD$; cf.\ \cite{GG:LS}.

\subsubsection{From local to filtered symplectic homology}
\label{sec:SH-loc}
Let $x$ be an isolated closed Reeb orbit, not necessarily prime. Then
we have the local symplectic homology of $x$, equivariant and
ordinary, defined. Referring the reader to, e.g.,
\cite{CGG:Reeb-HZ,GG:LS} for a much more detailed discussion, here we
only briefly review the definition and basic properties to the extent
used in the paper.

Fix an isolating neighborhood $U$ of $x(\R)$ in $M$, i.e., $U$ is a
neighborhood of $x$ containing no other closed orbit of period close
to $\CA(x)$.  Set $\UU=U\times (1-\eps,1+\eps)\subset \WW$ and let
$H=h(r)$ be a Hamiltonian on $\WW$ such that $h'(1)=\CA(x)$ and
$h''(1)>0$. Then $x$ gives rise to an isolated set $S$ of 1-periodic
orbits of $H$ with initial conditions on $x(\R)$. By definition, the
symplectic homology $\SH(x;\F)$ is the local Floer homology
$\HF(H;\F)$ of $H$ at $S$, and $\COH(x;\F)$ is the $S^1$-equivariant
local Floer homology.

For instance, when $x$ is non-degenerate and good and $m=\mu(x)$, we
have $\COH(x)=\F$ supported in degree $m$ and $\SH(x)$ is $\F$ in
degrees $m$ and $m+1$ and zero otherwise. If $x$ is bad, the same is
true when $\charr \F=2$, and the local homology is zero when
$\charr \F\neq 2$.

For local symplectic homology we have analogues of the Leray spectral
sequence, \eqref{eq:Leray}, with similar properties and the Gysin
exact sequence
\begin{equation}
  \label{eq:loc-Gysin}
  \dotsb \to \SH_m(x;\F)\to \COH_m(x;\F)
  \stackrel{\DD}{\longrightarrow}
  \COH_{m-2}(x;\F)\to \SH_{m-1}(x;\F)\to \dotsb .
\end{equation}
Moreover, the shift operator $\DD$ vanishes for the local homology
when $\charr \F=0$; \cite{GG:LS}.  (This is no longer true when
$\charr \F > 0$ unless $\charr \F$ and the order of iteration of $x$
are relatively prime.)  As a consequence,
\begin{equation}
\label{eq:SH-CH}
\SH(x;\Q)=\COH(x;\Q)\oplus \COH(x;\Q)[-1] .
\end{equation}

Local symplectic homology spaces are the building blocks for the
filtered symplectic homology. This principle can be formalized in
several ways.  For instance, assume that $\tau$ is the only point of
$\CS(\alpha)$ in an interval $I=[a,b)$ with $a>0$, and all closed Reeb
orbits with action $\tau$ are isolated. Then
$$
%\begin{equation}
%  \label{eq:local-filt}
\sum_{\CA(x)=\tau}\dim \SH_m(x) =\dim \SH^I_m(W)
%\end{equation}
$$
In fact,

\begin{equation}
\label{eq:sum}
  \bigoplus_{\CA(x)=\tau}\SH(x) =\SH^I(W)
\end{equation}
by the definition of local symplectic homology. Therefore, we have the
long exact sequence
  \begin{equation}
    \label{eq:long-exact}
   \dotsb \to\SH^a_m(W)\to \SH^b_m(W)
   \to\bigoplus_{\CA(x)=\tau}\SH_m(x)
   \to\SH^a_{m-1}(W)\to \dotsb .
 \end{equation}
 Moreover, for any field $\F$ and any interval $I$, there exists a
 spectral sequence
\begin{equation}
  \label{eq:spec-SH}
  E^r\Rightarrow \SH^I(W) \textrm{ with }
  E^1=\bigoplus_{\CA(x)\in I}\SH(x),
\end{equation}  
where, when $0\in I$, we treat $W$ as a collection of ``orbits'' with
action 0 as described above. This spectral sequence is simply
associated with the action filtration; see, e.g., \cite[Sec.\
2]{GG:LS}.

We have analogues of \eqref{eq:sum} and \eqref{eq:long-exact} for
equivariant symplectic homology and a spectral sequence
\begin{equation}
  \label{eq:spec-CH}
  E^r\Rightarrow \COH^I(W) \textrm{ with }
  E^1=\bigoplus_{\CA(x)\in I}\COH(x),
\end{equation}
where, for the sake of simplicity, we have assumed that $0\not\in I$.
As a consequence, when $W\subset \R^{2n}$ is a star-shaped domain, for
every $m\in n-1+2\N$ there must be a closed orbit $x$ with
$\COH_m(x)\neq 0$ due to \eqref{eq:CH}. Furthermore, if $q\in n+2\N$
and $\COH_q(y)\neq 0$ we must have
\begin{equation}
  \label{eq:odd}
  \sum\dim\COH_{q-1}(z)\geq 2 \quad \textrm{or} \quad
  \sum\dim\COH_{q+1}(z)\geq 2,
\end{equation}
where the summation is over all closed orbits $z$. This can also be
easily seen from \eqref{eq:CH} and the equivariant version of
\eqref{eq:long-exact}.

\subsection{Vanishing}
\label{sec:SH-gen}
In Section \ref{sec:spectral-def}, defining equivariant spectral
invariants in a sufficiently general setting, we will need the
following standard fact.

\begin{Theorem}[Vanishing]
\label{thm:vanishing}  
Let $V\to \WW$ be an exact symplectic embedding such that image of $V$
is displaceable in $\WW$. Then $\SH(V)=0$.
\end{Theorem}

Several variants of this vanishing theorem have been proved after the
original result in \cite{Vi}. To mention just a few without attempting
to be comprehensive, a theorem almost identical to Theorem
\ref{thm:vanishing} was established in \cite{CFO} under minor
additional conditions on $V$ via vanishing of Rabinowitz Floer
homology. A proof of Theorem \ref{thm:vanishing} for $V=W$ is
implicitly contained in \cite{Su}.  A direct proof of the theorem in
this case was also given in \cite{GS} in the more general setting
where $W$ is monotone. For the sake of completeness, we sketch the
proof of Theorem \ref{thm:vanishing} in Section \ref{sec:vanishing}.

The class of Liouville domains admitting exact symplectic embeddings
into, say, $\R^{2n}$ is much larger than the class of star-shaped
domains up to symplectomorphisms.

\begin{Example}
  \label{ex:exact}
  Let $L$ be a submanifold of dimension $k$ in $\R^{2n}$. Assume that
  $L$ is tangent to the kernel $\xi$ of the standard Liouville form
  $(p\, dq - q\, dp)/2$, not tangent to the Liouville vector field and
  not passing through the origin where the form vanishes. For
  instance, $L$ can be a Legendrian submanifold in $S^{2n-1}$. Then we
  can identify $TL\oplus T^*L$ with a symplectic subbundle of
  $\xi$. Assume furthermore that the symplectic orthogonal to
  $TL\oplus T^*L$ is trivial as a symplectic vector bundle. Then we
  can identify a neighborhood of $L$ with the product of a disk bundle
  in $T^*L$ and $B^{2(n-k)}$. Smoothing corners, we obtain an exact
  embedding of a Liouville domain into $\R^{2n}$, which is homotopy
  equivalent to $L$. At the same time, on a more subtle level, we do
  not know if there are exact embeddings $V\to \R^{2n}$ such that $V$
  is diffeomorphic to a ball, but the contact structure on
  $\p V\cong S^{2n-1}$ is not standard. This question might be
  intractable at the time of writing.
\end{Example}  

A simple but important consequence of the vanishing of symplectic
homology is an upper bound on the maximal length of a bar.

\begin{Theorem}[Bar length upper bound, \cite{GS}]
  \label{thm:depth}
  Assume that $\SH^\infty(W)=0$. Then all bars in the barcode of
  $\SH(W)$ are bounded by a constant $\Cbar \ge 0$ independent of the
  location of the bar. When $W$ is displaceable in $\WW$, the
  displacement energy can be taken as $\Cbar$, which is then also
  independent of the ground field $\F$.
 \end{Theorem}

 This theorem is \cite[Prop.\ 3.5]{GS}, which is essentially due to
 K. Irie, combined with \cite[Thm.\ 3.8]{GS}.

\subsubsection{Proof of Theorem \ref{thm:vanishing} -- outline} 
\label{sec:vanishing}
Following \cite{GS}, we will show that there exists a constant
$C\ge 0$ such that for every interval $I\subset \R$, the natural
``quotient-inclusion'' map $\SH^I(V)\to \SH^{I+C}(V)$ is zero. This
clearly implies that $\SH(V)=0$, but is actually equivalent to
vanishing of the total homology. (This observation is due to Irie; see
\cite{GS} for a proof.) Furthermore, it is not hard to see from the
long exact sequence that it suffices to prove vanishing of the
quotient-inclusion map for the intervals $I$ of the form
$(-\infty, a)$. In other words, we need to show that all maps
$\SH^a(V)\to \SH^{a+C}(V)$ are zero for some constant $C$ independent
of $a$. An examination of the argument below shows that we can take
any constant larger than the displacement energy of $V$ in $\WW$ as
$C$.

By scaling $W$ if necessary by the Liouville flow in $\WW$, we can
assume without loss of generality that we have an exact symplectic
embedding of a neighborhood $U$ of $V$ in $\VV$ into $W$ such that the
image of $U$ is displaceable in $W$. In what follows, we identify $U$
with its image in $W$.  Since the embedding is exact, after shrinking
$U$ slightly, we can adjust the Liouville form on $W$ to make it agree
with the Liouville form on $U$, which we assume from now on.

Denote by $\rho$ the standard coordinate on the cylindrical part
$\p V\times [1,\infty)$ of $\VV$. We can assume that $U\setminus V$
has the form $\p V\times (1,\rho_0)$ for some $\rho_0>1$. Likewise, we
denote by $r$ the coordinate on the cylindrical part
$\p W\times [1,\infty)$ of $\WW$.

Next, consider a Hamiltonian $H$ on $\WW$ with the following
properties.
\begin{itemize}
  
\item On $U$:
  
  \begin{itemize}
  \item $H\equiv 0$ on $V$,
  \item $H$ is a function of $\rho$ on $U\setminus V$, convex near
    $\rho=1$, concave near $\rho=\rho_0$, and linear in between with
    slope outside the action spectrum of $\p V$.
   \end{itemize}
   
 \item On $\WW\setminus U$:

   \begin{itemize}
   \item $H\equiv c>0$ on $W\setminus U$,
   \item $H$ is a non-strictly convex function of $r$ on
     $\p W \times [1,\infty)$, strictly convex near $r=1$, linear when
     $r>r_0$ for some $r_0>1$ with slope outside the action spectrum
     of $\p W$.
   \end{itemize}

 \end{itemize}
 Note that if $H$, thought of as a function of $\rho$ on
 $U\setminus V$, is sufficiently close to a linear function, and all
 non-constant 1-periodic orbits of $H$ in $\WW$ with action below $c$
 lie in the range $1<\rho<\rho_0$. It follows that for a given action
 threshold $a$, with a suitable choice of $H$ as above and, in
 particular, $c$ large, we have $\HF^a(H)=\SH^a(V)$; cf.\
 \cite{CGGM:Entropy}.

 Let $F$ be a Hamiltonian on $\WW$ supported in $W$, whose time-1 map
 $\varphi_F$ map displaces $U$ in $W$. Consider the composition
 $\varphi_K=\varphi_F\varphi_H$, where $K$ is the concatenation of $H$
 and $F$ after a suitable time-change. All 1-periodic orbits of $K$ in
 $W$ are 1-periodic orbits of $c+F$. Hence, these orbits have action
 above $c-c_0$, where $c_0$ is some constant which depends only on $F$
 but not on $H$. As a consequence, $\HF^{s}(K)=0$ for all $s<c-c_0$.

 The Hamiltonians $H$ and $K$ have the same slope at infinity in
 $\WW$.  Therefore, we have the standard continuation maps from
 $\varphi_H$ to $\varphi_K$ and back to $\varphi_H$. With our
 conventions, these maps shift the action as follows:
$$
\HF^a(H)\to \HF^{a+c_-}(K)\to \HF^{a+c_-+c_+}(H),
$$
where
$$
c_+=\int_0^1\max_W F_t\, dt \textrm{ and } c_- = \int_0^1- \min_W
F_t\, dt .
$$
Hence, the continuation map
$$
\SH^a(V)=\HF^a(H)\to \HF^{a+C}(H)=\SH^{a+C}(V)
$$
is zero when $a+c_-<c-c_0$ and $C\geq c_-+c_+$. It is easy to see that
for a suitable choice of $H$ this map is equal to the structure map
$\SH^a(V)\to \SH^{a+C}(V)$; see, e.g., \cite{CGGM:Entropy, GS}. This
concludes the proof of the theorem.\qed

\begin{Remark}
  In the choice of the Hamiltonian $H$ we followed tradition, but a
  simpler choice would work just as well. Namely, we could have
  required $H\equiv c$ entirely on $\WW\setminus U$. Then, by the
  maximum principle, the Floer homology $\HF^a(H)$ would still be
  defined and equal $\SH^a(V)$ when $c$ is large enough as in the
  proof, and the rest of the argument would go through word-for-word
  as above; cf.\ \cite{CGGM:Inv}.
\end{Remark}  

\subsection{Equivariant spectral invariants} 
\label{sec:spectral-def}
The definition of all spectral invariants follows the same standard
recipe going back in some form to, e.g., \cite{Vi:gen}. In the
equivariant Reeb framework, it takes the following form. Fix a
Liouville domain $W$ and a class $\zeta\in \COH^+(W)$. Given an exact
symplectic embedding $V\to \WW$, let $\cf_\zeta(V)\in (0,\infty)$ be
the infimum of all $a\in \R$ such that the Viterbo transfer of $\zeta$
to $\COH^+(V)$ is contained in the image of
$\COH^{(\eps,a)}(V)\to \COH^+(V)$. Then, as can be easily checked,
$\cf_\zeta(V)$ has the expected properties of a (contact, equivariant)
spectral invariant aka a symplectic capacity aka an action
selector. For instance,
\begin{itemize}
\item{Monotonicity:}
  $\cf_\xi$ is monotone with respect to inclusion of exact
  symplectic embeddings;

\item{Homogeneity:} $\cf_\zeta(\lambda V)=\lambda \cf_\zeta(V)$, where
  $\lambda V$ is the image of $V$ under the time-$(\ln\lambda)$
  Liouville flow on $\WW$.
\end{itemize}
In this construction, $\zeta$ is just a class in the symplectic
homology group which depends on $W$.  In some instances, it is
convenient to have spectral invariants which are universal, i.e.,
naturally present for a sufficiently interesting class of Liouville
domains $W$ and/or $V$. This is the approach taken in \cite{GH} and
below we briefly review the definition.

A natural class of Liouville domains to define the equivariant
spectral invariants $\cf_k$, $k\in\N$, is that of Liouville domains
$W$ with $\COH(W)=0$.

For such a Liouville domain $W$ and $k\in\N$, set
$$
\xi_k=[W]\otimes u^{k-1}\in \H_{*}(W,\p W)[u]\cong \COH^-(W),
$$
where we use identification \eqref{eq:COH-}. This is a class of
degree $n+2(k-1)$. The connecting map
$\delta\colon \COH^+(W)\to \COH^-(W)$ in the long exact
sequence, \eqref{eq:COH-long}, is an isomorphism of degree $-1$.  Let
\begin{equation}
  \label{eq:clases}
\zeta_k:=\delta^{-1}(\xi_k)\in \COH^+_{n+1+2(k-1)}(W).
\end{equation}

Now, applying the above construction to $\zeta=\zeta_k$, we obtain a
sequence of spectral invariants $\cf_k(W)$. (Labeling these invariants
starting with $k=1$ rather than $k=0$ is awkward but common.)

More directly, by the long exact sequence we have
\begin{equation}
  \label{eq:xi-spec}
\cf_k(W)=\inf\{c>0\mid j_-^c(\xi_k)=0\},
\end{equation}
where $j_-^c\colon \COH^-(W)\to \COH^c(W)$ is the structure map. In
other words, for every $m=n+(k-1)$, $k\in \N$, there is a bar
beginning at $\xi_k$ at zero level in the equivariant symplectic
homology persistence module, and $\cf_k(W)$ is the end-point of that
bar. This is the perspective that we will mainly take in this paper.

It is easy so see that an exact embedding $V\to W$, where we also
require that $\COH(V)=0$, pushes back the classes $\xi_k$ and
$\zeta_k$ for $W$ to their counterparts for $V$. As a consequence, the
resulting spectral invariants satisfy the above monotonicity, i.e.,
$\cf_k(V)\leq \cf_k(W)$, and homogeneity conditions.

Moreover, these spectral invariants enjoy an additional monotonicity
property. Namely,
\begin{equation}
\label{eq:monotone}  
\cf_1(W)\leq \cf_2(W)\leq \cf_3(W)\leq \dotsb .
\end{equation}
The reason is that $\DD\xi_{k+1}=\xi_k$ for $k\in\N$ and hence
$\DD\zeta_{k+1}=\zeta_k$, where $\DD$ is the shift map in the Gysin
exact sequence, and the spectral values are obviously decreasing under
$\DD$. Moreover, the inequalities in \eqref{eq:monotone} are strict
when all closed Reeb orbits on $\p W$ are isolated; \cite{GG:LS}.

The condition that the total equivariant symplectic homology of $W$
vanishes is difficult to verify directly, and it can be replaced by
the requirement that $\SH(W)=0$, which is stronger than that
$\COH(W)=0$. (However, we are not aware of any examples where
$\COH(W)=0 $ but $\SH(W)\neq 0$.) Furthermore, the vanishing
requirement on $\SH(W)$ can be further replaced by the displaceability
condition as in Theorem \ref{thm:vanishing}, which is usually much
more tractable.

So far, the ground ring has played no role in the
construction. However, the spectral invariants actually depend on the
ring as Theorem \ref{thm:spectral} and Example \ref{ex:ellipsoids}
show. It is not hard to see that when the ground ring is a field, the
equivariant spectral invariants are completely determined by its
characteristic $p$, and we denote the resulting spectral invariants by
$\cfp_i$. In particular, we have the spectral invariants $\cf_i^{(0)}$
when $\F=\Q$ or more generally $\charr\F=0$. One of the points of this
paper is to initiate the study of the positive characteristic
equivariant spectral invariants $\cfp_i$ for $p>0$. We note in this
connection that by the universal coefficient formula and
\eqref{eq:xi-spec},
$$
\cfp_i(W)\leq \cf_i^{(0)}(W)
$$
and $\cfp_i(W)=\cf_i^{(0)}(W)$ as along as
$\cfp_i(W)< \min p\CS(\p W)$, i.e., when $\cfp_i(W)$ is below the
smallest action of a $p$-iterated orbit; cf.\ Example
\ref{ex:ellipsoids}.

\begin{Remark}[Other $S^1$-equivariant spectral invariants]
  Equivariant spectral invariants of the type considered here are
  defined in some other situations even when $\COH(W)\neq 0$. For
  instance, let $W$ be the unit disk cotangent bundle to $S^{n}$. The
  connecting map $\delta$ is injective, i.e.,
  $\COH(W)\cong \COH^-(W)\oplus \COH^+(W)$, and the above construction
  breaks down. Yet, there is a sequence of classes
  $\zeta_k\in \COH^+(W)$ playing roughly the same role as the classes
  defined by \eqref{eq:clases} and such that \eqref{eq:monotone} holds
  at least for some sufficiently long finite subsequences; see
  \cite{GG:LS} and references therein. (We do not know if in this case
  \eqref{eq:monotone} holds literally as stated.)
\end{Remark}

\section{Proof of Theorem \ref{thm:spectral}}
\label{sec:proof-spectral}
Throughout this section, we will assume that the ground ring is a
field $\F$ but we will continue suppressing it in the notation when
$\charr \F$ is inessential.

Clearly, the sequence of spectral invariants, \eqref{eq:spec-seq}, is
bounded for one domain if and only if it is bounded for all domains
$W$. Hence it suffices to show that it is bounded when $W$ is an
ellipsoid. For $a\leq b$, denote by
$$
j_a^b\colon \COH^a(W) \to \COH^b(W) 
$$
the equivariant structure map; see Section \ref{sec:SH-def}.

\begin{Proposition}
  % \begin{Corollary}
  \label{prop:ellipsoid}
  Assume that $\p W$ is an irrational ellipsoid and that the interval
  $I=[a,b) \subset \R_{>0}$ contains a Reeb period, i.e.,
  $I\cap \CS(\alpha)\neq 0$. Then
  \begin{equation}
    \label{eq:ellipsoid-collapse0}
        j_a^b=0
    \end{equation}
    if and only  if 
    \begin{equation}
      \label{eq:ellipsoid-collapse}
        \COH^I(W) = \SH^{I}(W)[u]
    \end{equation}
    as graded vector spaces, i.e., the Leray spectral sequence
    collapses at the $E^1=E^2$-page.  More specifically, over a field
    of characteristic $p$ these conditions are satisfied if and only
    if there is a $k$-iterated closed Reeb orbit with action in $I$
    and $k$ divisible by $p$.
%\end{Corollary}
  \end{Proposition}

  Even though the proposition is specific to irrational ellipsoids,
  the statement is non-obvious: whenever $k$ is divisible by $p$, an
  iterate $x^k$ makes an infinite contribution to the equivariant
  symplectic homology. Whether or not conditions
  \eqref{eq:ellipsoid-collapse0} and \eqref{eq:ellipsoid-collapse}
  hold for $I$ depends on $\charr \F$. The second part of the
  proposition gives an explicit criterion when these conditions are
  met. When $\charr \F=0$, these conditions are never satisfied. This
  follows from the proposition, but a simple reason is, for instance,
  that for every $b>0$ the image of the class $\xi_k$ in
  $\COH^b(W;\Q)$ is non-zero for all large $k\in\N$.

  Theorem \ref{thm:spectral} readily follows from the proposition and
  \eqref{eq:xi-spec}. Proposition \ref{prop:ellipsoid} is in turn a
  consequence of the following lemma which holds for all Liouville
  domains $W$ with $c_1(TW)=0$ and all ground fields.

\begin{Lemma}
\label{lem:key}
    Assume that for some $0<a<b$ the structure map 
    \begin{equation}
    \label{eq:lemma_1}
      i_a^b\colon  \SH^a(W) \to \SH^b(W)
    \end{equation}
     is identically zero and that
    \begin{equation}
        \label{eq:lemma_2}
        \COH^I(W) = \SH^{I}(W)[u]
    \end{equation}
    for $I=[a,b)$ as graded spaces. (In other words, the Leray
    spectral sequence collapses at the $E^1=E^2$-page). Then the
    equivariant structure map
    \begin{equation}
    \label{eq:lemma_3}
     j_a^b\colon   \COH^a(W) \to \COH^b(W)
    \end{equation}
    is identically zero, and we also have
    \begin{equation}
        \label{eq:lemma_4}
        \COH^a(W) = \SH^{a}(W)[u] \textrm{ and }
        \COH^b(W) = \SH^{b}(W)[u]
    \end{equation}
    as graded vector spaces. Conversely, if the map \eqref{eq:lemma_3}
    is zero and \eqref{eq:lemma_4} holds, then \eqref{eq:lemma_1} is
    zero and \eqref{eq:lemma_2} holds.
\end{Lemma}

\begin{proof}[Proof of Proposition \ref{prop:ellipsoid}]
  To prove equivalence of the two conditions, note first that for
  ellipsoids or more generally whenever $\dim \SH^a(W)=1$ for all
  $a>0$, the map $i_a^b$ in \eqref{eq:lemma_1} vanishes if the
  interval $I=[a,b)$ contains a Reeb period; cf.\
  \cite{CGG:Reeb-HZ}. Furthermore, for an irrational ellipsoid,
  \eqref{eq:lemma_4} holds automatically for any $0<a<b$ or
  equivalently for all $a>0$, i.e., the spectral sequence collapses at
  the $E^1=E^2$-page. Indeed, for an irrational ellipsoid, $\SH^a(W)$
  is one-dimensional for every field $\F$ and supported in the degree
  $m$ equal to one plus the Conley--Zehnder index of the closed Reeb
  orbit with the largest action smaller than $a$. Thus
  $E^1=E^2=\SH^a(W)[u]$ is supported in degrees $m,\, m+2,\, \dotsc$
  and is one-dimensional in each of these degrees. Therefore, either
  the spectral sequences collapses at the $E^1=E^2$-page or
  $\COH^a_*(W)$ vanishes in at least one of these degrees. However, as
  is well known, $\COH^a_*(W;\Q)\neq 0$ for $*=m,\, m+2,\, \dotsc$;
  see, e.g., Example \ref{ex:ellipsoids}. Since the equivariant
  symplectic homology is defined over $\Z$, by the universal
  coefficient formula $\COH^a_*(W;\F)\neq 0$ for all $\F$ and
  $*=m,\, m+2,\, \dotsc$\ .  (As we have already pointed out,
  $\COH^a(W)$ and $E^r$ are modules over $\H^*(BS^1)=\F[u^{-1}]$, but
  not over $\H_*(BS^1)=\F[u]$.  These modules are entirely torsion and
  not finitely generated. Hence, we need to argue degree-by-degree
  rather than comparing the ranks of these modules.)  To summarize,
  \eqref{eq:lemma_1} and \eqref{eq:lemma_4} are satisfied
  automatically for irrational ellipsoids, and therefore
  \eqref{eq:ellipsoid-collapse0} and \eqref{eq:ellipsoid-collapse} are
  equivalent.

  Next, let us show that conditions \eqref{eq:ellipsoid-collapse0}
  and/or \eqref{eq:ellipsoid-collapse} are satisfied if and only if
  there is a $k$-iterated closed Reeb orbit with action in $[a, b)$
  and $k$ divisible by $p$. Clearly, by
  \eqref{eq:ellipsoid-collapse0}, it is enough to prove this when
  there is only one closed orbit $x$ with action in $I$. Let $k$ be
  its iteration order. Then
$$
\COH^I(W) = \SH^{S^1}(x) \textrm{ and } \SH^{I}(W) = \SH(x),
$$
where on the right in both identities we have the local symplectic
homology of $x$. Up to a shift of degrees by the Conley--Zehnder index
of $x$, in both the equivariant and non-equivariant settings, the
local symplectic homology of $x$ is canonically isomorphic to the
Morse homology of $S^1$, where in the equivariant case $S^1$ acts on
itself with stabilizer $\Z_k$; cf.\ \cite{GG:LS}. (Here, in
particular, we are implicitly using the fact that $x$ is
non-degenerate, elliptic and hence good.) Let now $\F=\F_p$. Then
$$
\COH(x;\F_p)=\H_*\big(S^1\times_{S^1}
ES^1;\F_p\big)=\H_*(ES^1/\Z_k;\F_p)=\H_*(B\Z_k;\F_p).
$$
When $k$ is divisible by $p$, we have 
$$
\H_*(B\Z_k;\F_p)=\H_*(S^1;\F_p)[u]=\SH(x;\F_p)[u].
$$
In other words, the Leray spectral sequence collapses and
\eqref{eq:ellipsoid-collapse} holds. On the other hand, when $p$ and
$k$ are relatively prime, $\H_*(B\Z_k;\F_p)$ is supported in degree
$0$ and hence \eqref{eq:ellipsoid-collapse} fails.  This completes the
proof of the proposition.
\end{proof}

It remains to prove the lemma.

\begin{proof} [Proof of Lemma \ref{lem:key}]
For any $0<a<b$ and $m \in \N$, we have
\begin{align}
  \label{eq:lemma_11}
    \dim \SH^I_m(W) &\leq  \dim \SH^{b}_m(W) + \dim \SH^{a}_{m-1}(W),
  \\
  \label{eq:lemma_33}
    \dim \COH^I_m(W) &\leq  \dim \COH^b_m(W) + \dim \COH^a_{m-1}(W).
    \end{align}
    Furthermore, in both cases we have equality if and only if the
    maps \eqref{eq:lemma_1} and \eqref{eq:lemma_3} are identically
    zero, respectively. This point is central to the argument
    below. Note also that in general
\begin{equation*}
  \dim \COH^{\diamondsuit}_m(W) \leq
  \dim (\SH^{\diamondsuit}(W)[u] \big )_m,
\end{equation*}
where $\diamondsuit=a$ or $b$ or $I$. This is a consequence of the
Leray spectral sequence.  With these observations in mind, we are in a
position to prove the lemma. Assume that the map \eqref{eq:lemma_1}
vanishes and \eqref{eq:lemma_2} holds. Then
\begin{align*}
  \dim \COH^I_m(W) &\leq  \dim \COH^b_m(W) + \dim \COH^a_{m-1}(W) \\
                   &\leq  \dim \big( \SH^{b}(W)[u] \big )_m
                     + \dim \big (\SH^{a}(W)[u] \big)_{m-1}  \\
                   &= \dim \big( \SH^{I}(W)[u] \big)_m \\
                   &=\dim \COH^I_m(W).
\end{align*}
Thus, \eqref{eq:lemma_33} turns into equality when \eqref{eq:lemma_1}
vanishes and \eqref{eq:lemma_2} holds, and hence the map
\eqref{eq:lemma_3} also vanishes.

Conversely, assume that the map \eqref{eq:lemma_3} vanishes and
\eqref{eq:lemma_4} holds. Then
\begin{align*}
  \dim \big ( \SH^{I}(W)[u] \big )_m
  &\leq \dim \big ( \SH^{b}(W)[u] \big )_m +
    \dim \big (\SH^{a}(W)[u] \big)_{m-1}  \\
  &= \dim \COH^b_m(W) + \dim \COH^a_{m-1}(W) \\
  &= \dim \COH^I_m(W) \\
  &\leq \dim \big ( \SH^{I}(W)[u] \big )_m,
    \end{align*}
    and hence \eqref{eq:lemma_11} is equality. We conclude that the
    map \eqref{eq:lemma_1} vanishes and \eqref{eq:lemma_2} holds. This
    completes the proof of the lemma.
\end{proof}

\section{Proofs of Theorems \ref{thm:negative} and \ref{thm:bounded}}
\label{sec:negative}
\subsection{Bounded symplectic homology: the proof of Theorem \ref{thm:bounded}}
\label{sec:bounded-pf}
We start with the proof of Theorem \ref{thm:bounded}, which is an easy
consequence of already known results and implicitly contained in
\cite{CGG:Reeb-HZ}, as mentioned earlier. Since the role of the field
is inessential, we suppress it in the notation.

Recall first that for $a\in\R$ the number $\zeta(a)$ of bars beginning
or ending at $a$ is
\begin{equation}
  \label{eq:zeta}
\zeta(a)=\sum_{\CA(y)=a}\dim \SH(y),
\end{equation}
where the sum is taken over all closed Reeb orbits $y$ with action
$a$. This readily follows from the long exact sequence; see
\cite[Lemma 2.15]{CGG:Reeb-HZ} for a more precise result. (Observe
that both sides are zero if $a\notin \CS(\alpha)$.)

Next, note that $\dim \SH^t(W)$ is the number of bars passing through
$t$. Assuming that the flow is a Reeb pseudo-rotation, denote the
prime closed Reeb orbits by $x_1,\dotsc, x_r$ and their actions by
$a_1\leq \dotsb \leq a_r$, where we have arranged the orbits in the
order of increasing action. By Theorem \ref{thm:depth}, every bar
passing through $t$ has length less than $C$ for some constant $C$
completely determined by $W$ and the field $\F$. Hence, by
\eqref{eq:zeta},
\begin{equation}
  \label{eq:zeta_upper_bound}
\dim \SH^t(W)\leq\frac{2rC}{a_1}\max_{x_i^k}\dim \SH\big(x_i^k\big).
\end{equation}

For every orbit $x$ such that all iterates $x^k$ are isolated,
$\dim \SH\big(x^k\big)$ is bounded from above by some constant $d(x)$
depending only on $x$. This is proved in \cite[Thm.\ 3.1]{McL}; see
also \cite{HM} for a similar statement for local contact homology and
\cite{GG:gap,GrMe1,GrMe2} for previous closely related results. This
gives an upper bound on the right-hand side of
\eqref{eq:zeta_upper_bound} and completes the proof of the
theorem. \qed

\subsection{Beginning and end maps}
\label{sec:beg-end}
Turning to the proof of Theorem \ref{thm:negative}, we need to review
some background results. As in the theorem, let $W$ be a star-shaped 
domain in $\R^{2n}$.  Throughout the section we will assume that the
Reeb flow on $\p W$ is non-degenerate and denote by $\PP$ the set of
closed Reeb orbits visible over the ground field $\F$. In other words,
due to non-degeneracy, $\PP$ comprises all closed Reeb orbits when
$\charr \F=2$ and all good orbits otherwise.  Furthermore, let us also
formally add to $\PP$ an extra element corresponding to the domain
itself denoting the resulting set by $\PP'$. We assign zero action and
index $n$ to this element and think of it as ``beginning'' of the
unique bar beginning at action zero.

To prove the theorem, we need the notion of the beginning and end maps
of the barcode $\CB$ of the symplectic homology persistence module.

\begin{Theorem}[\cite{CGG:Reeb-HZ}]
  \label{thm:beg-end}
  Assume that the Reeb flow on the boundary $\p W$ of a star-shaped
  domain $W$ is non-degenerate.  There exist maps
  $\beg\colon \CB\to \PP'$ (the beginning of $I$) and
  $\en\colon \CB\to \PP$ (the end of $I$) such that for every bar
  $I=(a,b]$ and $x=\beg(I)$ and $y=\en(I)$ we have
\begin{itemize}
\item $\CA(x)=a$ and $\CA(y)=b$,
\item $\deg(I)$ is either $\mu(x)$ or $\mu(x)+1$ and $\deg(I)+1$ is
  either $\mu(y)$ or $\mu(y)+1$.
\end{itemize}
\end{Theorem}
This theorem is proved in \cite[Thm.\ 2.16]{CGG:Reeb-HZ} in a more
general form without non-degeneracy assumption and for all Liouville
domains. As stated, it also readily follows from the existence of a
singular value decomposition; see \cite{UZ}.

\begin{Remark}
  \label{rmk:beg-end}
  The maps $\beg$ and $\en$ are unique when all closed Reeb orbits
  have distinct action values.  Without this assumption, however,
  their construction involves certain choices, and uniqueness is lost.
  Furthermore, the number of bars beginning or ending at an action
  value $a$ is twice the number of closed Reeb orbits with action
  $a$. This fact readily follows from \eqref{eq:zeta} and
  non-degeneracy.  As a consequence, both of these maps are surjective
  at $a$ if and only they are both injective at $a$. Thus, every
  $x\in\PP$ with action $a$ is the beginning and the end of exactly
  one bar when $\beg$ and $\en$ are both injective at $a$. We see no
  reason why the maps would be surjective or injective in general.
\end{Remark}

For an interval $I\subset \R$, denote by $\CB_I^-$ and $\CB_I^+$ the
sets of bars beginning and, respectively, ending in $I$.  Theorem
\ref{thm:negative} is an easy consequence of Part \ref{i2} of the
following technical result proved in Sections \ref{sec:pf-beg-end}
and \ref{sec:pf-pert}.

\begin{Theorem}
%\begin{Proposition}
\label{prop:beg-end}
Assume that the Reeb flow on the boundary $M^{2n-1} \subset \R^{2n}$
of a star-shaped domain $W$ is non-degenerate and has finitely many
prime closed orbits. Then for any $C>0$, there exists an interval
$I=[a,b]$ of length greater than $C$, with endpoints outside the
action spectrum, and a choice of beginning/end maps
$\Beg\colon \CB \to \Pp'$ and $\End\colon \CB \to \Pp$ such that
    \begin{itemize}
    \item[\reflb{i1}{\rm{(i)}}]
      $\dim \SH^t(W; \F_2) = \max_{s>0} \dim \SH^s(W; \F_2)$ for all
      $t \in I$ and, in particular, $\dim \SH^t(W; \F_2) = \const$ for
      all $t \in I$;
    \item[\reflb{i2}{\rm{(ii)}}] the maps $\Beg$ and $\End$ are
      injective on $\CB_I^-$ and, respectively, on $\CB_I^+$.
    \end{itemize} 
%\end{Proposition}
  \end{Theorem}

  \begin{Remark}
\label{rmk:beg-end2}
Observe that \ref{i2} follows from \ref{i1} when all orbits have
distinct actions and thus the maps $\Beg$ and $\End$ are unique and
clearly one-to-one; cf., Remark \ref{rmk:beg-end}. (It suffices to
assume only that $\dim\SH^t(W;\F_2)=\const$ for $t\in I$.)  In
general, when there are two orbits with the same action, depending on
their indices, there could be an ambiguity in the choice of
beginning/end maps. Then \ref{i2} ensures that, in the setting of
Theorem \ref{prop:beg-end}, one can make this choice so that the
behavior is similar to the case of distinct actions. We note that this
is not a purely algebraic statement, and the proof heavily relies on
the assumption that there are finitely many prime orbits.
\end{Remark}

\begin{proof}[Proof of Theorem \ref{thm:negative}]
  Let $\Beg \colon \Bb \to \Pp'$ and $\End \colon \Bb \to \Pp$ be as
  in Theorem \ref{prop:beg-end}. Set
$$
\Pp_I:=\{ z \in \PP \, \vert \, \CA(z) \in I\}. 
$$
Since $\F=\F_2$, all orbits in $\PP_I$ are $\F$-visible. By Remark
\ref{rmk:beg-end} and \ref{i2},
every orbit in $\PP_I$ is the beginning of exactly one
bar and the end of exactly one bar:
$$
\Pp_I= \Beg (\CB_I^-) = \End (\CB_I^+). 
$$
Let us partition $\PP_I$ into classes $\PP_I^i$ of closed orbits where
each class comprises closed orbits connected by a sequence of
bars. More formally, the construction is as follows.

Let $x \in \PP_I$ be one of the lowest-action orbits. We put
$x \in \PP_I^1$. If the bar $\Beg^{-1}(x)$ is not contained in the
interval $I$, then $x$ is the only orbit in $\Pp_I^1$. If it is, then
we put the orbit $y:=\End(\Beg^{-1}(x)) \in \Pp_I^1$ as well. Next, we
ask whether or not $\Beg^{-1}(y) \subset I$.  If not, then $x$ and $y$
are the only orbits in $\Pp_I^1$. If yes,
$\End(\Beg^{-1}(y)) \in \Pp_I^1$, too. We continue until this process
terminates, i.e., until the end of a bar lies outside $I$. We then
define $\Pp_I^2$ similarly, starting from one of the lowest-action
orbits in $\PP_I\setminus\Pp_I^1$, and so on. (By \ref{i1}, this
recipe yields a partition of $\PP_I$ into exactly
$\max_{s>0} \dim \SH^s(W; \F_2)$ sets $\Pp_I^i$, although we do not
need this fact.) We note the following important property of the
partition: For any $x, y \in \Pp_I^i$, $\mu(x) \geq \mu(y)$ whenever
$\CA(x)>\CA(y)$.

Now, arguing by contradiction, assume that there exists a closed orbit
$z$ with $\mu(z)<0$. Fix an integer $N \in \N$ large enough so that
$\mu(z^k) < \mu(z^l)$ when $k-l > N$. By the pigeon hole principle, if
$C>0$ is sufficiently large, there always exist a class $\Pp_I^i$
containing two iterates $z^k$ and $z^l$ of $z$, which are
$N$-apart. This is a contradiction, since $\CA(z^k)>\CA(z^l)$ but
$\mu(z^k) < \mu(z^l)$.
\end{proof}

\subsection{Proof of Theorem \ref{prop:beg-end}}
\label{sec:pf-beg-end}
The key to the proof is the following result, asserting roughly
speaking that by a small perturbation of a Hamiltonian we can make
actions distinct without increasing the dimension of the filtered
Floer homology. In Lemma \ref{lem:perturbation} and Corollary
\ref{cor:perturbation}, the ground field $\F$ is arbitrary and
suppressed in the notation.

\begin{Lemma}[Perturbation Lemma]
\label{lem:perturbation}
Let $H \colon \widehat{W} \to \R$ be an admissible Hamiltonian and $x$
be a non-degenerate (as a Reeb orbit) 1-periodic orbit of $H$ whose
Hamiltonian action $\A_H(x)$ is isolated in the action spectrum
$\Ss(H)$ of $H$. Let $U \subset \WW$ be an isolating neighborhood of
$x$ and $f \colon \widehat{W} \to \R$ be a non-negative function
supported in $U$, which is constant and strictly positive near
$x$. Set
$$
K^{\pm} = H\pm sf,
$$
where $s\geq 0$. Then, for all sufficiently small $s\geq 0$, we have
$$
\dim \HF^t(K^{\pm}) \leq \dim \HF^t(H)
$$
for all $t\in \R$ and at least one of the Hamiltonians $K^{\pm}$.
\end{Lemma}

The lemma is proved in Section \ref{sec:pf-pert}.  As an immediate
consequence, we obtain the following global perturbation result.

\begin{Corollary}
\label{cor:perturbation}
Let $H \colon \widehat{W} \to \R$ be an admissible Hamiltonian such
that all non-constant 1-periodic orbits are non-degenerate (as Reeb
orbits).  Fix a neighborhood $U \subset \widehat{W}$ of the set
$\Pp(H)$ of such orbits. There exists an arbitrarily small autonomous
perturbation $K$ of $H$ supported in $U$ and such that
$\Pp(K)= \Pp(H)$, the action $\A_K \colon \Pp(K) \to \R$ is injective,
and
$$
\dim \HF^t(K) \leq \dim \HF^t(H)
$$
 for all $t \in \R$.
\end{Corollary}

\begin{proof}[Proof of Theorem \ref{prop:beg-end}]
  By Theorem \ref{thm:bounded}, the filtered symplectic homology
  $\SH^t(W;\F_2)$ has a maximal dimension $D$. Pick an interval
  $[a',b']$ outside the action spectrum such that
  $\dim \SH^t(W; \F_2) = D$ for all $t \in [a', b']$. Then we also
  have $\dim \SH^{t}(W; \F_2) = D$ for all
  $t \in [a,b]:=[2^ka', 2^kb']$ and all $k\in\N$ by the Smith
  inequality; see \cite{Se} and also \cite{ShZ}. In what follows, we
  will fix $k$ so large that the interval $I=[a,b]$ has length greater
  than the constant $C$ from the theorem: $2^k(b'-a')>C$. This
  completes the proof of \ref{i1}.

  Let now $H$ be an admissible Hamiltonian with slope
  $\lambda$. Recall that the Floer homology persistence module
  $\HF^t(H;\F_2)$ is a reparametrization of the persistence module
  $\SH^t(W; \F_2)$ for $t\leq \lambda$; see \cite[Thm.\
  3.5]{{CGGM:Entropy}}. Moreover, the reparametrization function for
  the Hamiltonian $sH$ converges to the identity as $s\to\infty$
  uniformly on compact sets in $[0,\infty)$; see \cite[Sec.\
  2.2]{{CGGM:Entropy}} and also \cite{CGGM:Inv}. As a consequence, for
  a suitable choice of $H$, we have
  $$ \dim \HF^t(H; \F_2) = D\textrm{ for all } t \in [a', b'],
  $$
  where $b' \leq \slope(H)$, and
  $$
  \dim \HF^{t}(2^kH; \F_2) = D \textrm{ for all }
  t \in [a,b]:=[2^ka', 2^kb'] \textrm { and all }k\in\N.
  $$
  Note that the persistence module $\HF^t(2^kH;\F_2)$ is a
  reparametrization of the symplectic homology persistence module as
  long as $t \leq 2^k\slope(H)$, and hence for $t \leq b$.
  Furthermore, we can ensure that the 1-periodic orbits of $2^kH$ with
  action in $I=[a,b]$ are in one-to-one correspondence with closed
  Reeb orbits with action in $I$. Therefore, to prove \ref{i2} it
  suffices to construct beginning and end maps injective on $I$ for
  the Floer homology persistence module $\HF^{t}(2^kH; \F_2)$.
  
  To this end, it suffices to find an arbitrarily small autonomous
  perturbation $K$ of $2^k H$ such that there is a one-to-one
  correspondence between closed orbits of $H$ and $K$, all closed
  orbits of $K$ have distinct actions, and
$$
\dim \HF^t(K; \F_2) =D
$$
for all $t \in I$.  Then both \ref{i1} and \ref{i2} hold for $K$ and
$I$; see Remark \ref{rmk:beg-end2}. Finally, we can define the
beginning/end maps for $H$ as the limit of the ones for such
perturbations $K_i \to H$. (Since there are finitely many 1-periodic
orbits with action in $I$, one can always find a constant
subsequence.)

As the final step, let us show that the required perturbations do
exist. Let $K$ be a perturbation of $2^k H$ as in Corollary
\ref{cor:perturbation}. If necessary, we choose a smaller perturbation
so that we have
$$
\dim \HF^t(1/2^k K ; \F_2)=D
$$
for all $t \in [a',b']$. Then the Smith inequality implies that
\begin{equation}
\label{eq:prop_beg/end}
\dim \HF^t(K ; \F_2) \geq D
\end{equation}
for all $t\in I$. Finally, we use Corollary \ref{cor:perturbation} to
conclude that \eqref{eq:prop_beg/end} is actually an equality. This
finishes the proof of Theorem \ref{prop:beg-end}.
\end{proof}
  
\subsection{Proof of Lemma \ref{lem:perturbation}}
\label{sec:pf-pert}
Throughout the proof we suppress the field $\F$ in the notation.  The
statement clearly holds for $s=0$. First, note that for a sufficiently
small $s>0$ all three Hamiltonians $H$ and $K^{\pm}=H\pm sf$ have the
same 1-periodic orbits as subsets of $\widehat{W}$. Furthermore, fix a
small $\eps>0$ and also pick a sufficiently small $s$ to guarantee
that $\eps> 2s \max f >0$ and the orbit $x$ is the unique 1-periodic
orbit of $K^{\pm}$ with
\begin{align*}
     \A_{K^{+}}(x) &\in (\A_{H}(x), \A_{K^{+}}(x) + \eps), \\
      \A_{K^{-}}(x) &\in (\A_{K^{-}}(x) - \delta, \A_{H}(x))
\end{align*}
for some $\delta>0$. Here the uniqueness is a consequence of
the assumption that the action value $\A_H(x)$ is isolated in
$\Ss(H)$. Next, we choose a specific non-degenerate perturbation
$\widetilde{H}$ of $H$.

\begin{Claim}
  There exists a non-degenerate (time-dependent) perturbation
  $\widetilde{H}$ of $H$ which satisfies the following conditions:
    \begin{itemize}
    \item[\reflb{ii1}{\rm{(i)}}] All three Hamiltonians
      $\widetilde{H}$, $\widetilde{K}^{\pm}:=\widetilde{H} \pm sf$
      have the same 1-periodic orbits.

    \item[\reflb{ii2}{\rm{(ii)}}] The orbit $x$ splits into exactly
      two non-degenerate orbits $x_1$ and $x_2$, and these orbits are
      contained in the region where $f$ is constant.

    \item[\reflb{ii3}{\rm{(iii)}}] The remaining 1-periodic orbits
      $y_i$ lie outside the support of $f$.

    \item[\reflb{ii4}{\rm{(iv)}}] For some $\delta'>0$ and
      $\eps'> 2s \max f$, the orbits $x_i$ are the only 1-periodic
      orbits of $\widetilde{K}^{\pm}$ with action
      \begin{align*}
        \A_{\widetilde{K}^{+}}(x_i) \in
        (\min \A_{\widetilde{K}^{+}}(x_i) - \delta', \max
        \A_{\widetilde{K}^{+}}(x_i) + \eps'),\\
        \A_{\widetilde{K}^{-}}(x_i) \in
        (\min \A_{\widetilde{K}^{-}}(x_i) -
        \delta', \max  \A_{\widetilde{K}^{-}}(x_i) + \delta').
      \end{align*}
      
    \item[\reflb{ii5}{\rm{(v)}}] Assume that $\HF(x) \neq 0$. Then the
      global continuation map
         $$
         \psi \colon \CF(\widetilde{K}^+) \to \CF(\widetilde{K}^-)
         $$
         induced by the linear homotopy $s \to -s$ is a chain-level
         isomorphism.
    \end{itemize}
      \end{Claim}
      We note that these conditions are not logically independent. For
      example, \ref{ii2} and \ref{ii3} imply \ref{ii1}. Among
      \ref{ii1}--\ref{ii4}, only the first part of \ref{ii2} is
      non-trivial but standard. The last property, \ref{ii5}, is
      crucial for the proof. Below we show that it follows from the
      remaining ones.

      To this end, observe first that strictly speaking the complexes
      $\CF(\widetilde{K}^{\pm})$ might depend on the underlying almost
      complex structure $J$. We choose a generic one and use the same
      almost complex structure for both Hamiltonians
      $\widetilde{K}^{\pm}$. The proof of \ref{ii5} is independent of
      the choice of $J$ and we drop it from the notation.

    \begin{proof}[Proof of  \ref{ii5}]
      Note that, by \ref{ii1}, it suffices to show that
      $\ker \psi =0 $. Furthermore, observe that, by \ref{ii3}, the
      linear homotopy is constant near $y_i$. Hence, we have
     
    \begin{equation}
    \label{eq:constant_solution}
       \psi(y_i) = y_i + \dotsb,
    \end{equation}
    where the dots stand for the terms with action strictly lower than
    the action of $y_i$.
    
    Pick an arbitrary non-zero $z \in \CF(\widetilde{K}^+)$. There are
    two cases to deal with. Assume first that the leading action terms
    of $z$ or leading action terms of its image under the boundary map
    $d_{\widetilde{K}^+} ( z)$ belong to the set $\{ y_i\}$. (Here, in
    $d_{\widetilde{K}^+} ( z)$, we only consider the generators with
    non-zero coefficient in $\F$.) Then, by
    \eqref{eq:constant_solution} and since $\psi$ is a chain map, we
    have $z \notin \ker \psi$. In the second case, we assume that the
    leading action terms of $z$ belong to $\{x_i\}$ and
    $ d_{\widetilde{K}^+} (z) =0$. Note that, since $\HF(x) \neq 0$,
    these two cases cover all possibilities. More precisely, if the
    leading action terms of $z$ belong to $\{x_i\}$, then either
    $d_{\widetilde{K}^+} (z) = 0$ or, since
    $\dim \HF(x) = \vert \{ x_i \} \vert$, its leading action terms
    belong to $\{y_i\}$; and the second possibility is already covered
    by the first case.
    
    Next, recall that, since the Hamiltonians $\widetilde{K}^{\pm}$
    have the same slope at infinity, there is a continuation map going
    in the opposite direction
    $$
    \psi' \colon \CF(\widetilde{K}^-) \to \CF(\widetilde{K}^+),
    $$
    which increases the action filtration by at most $2s\max
    (f)$. Choose a filtration level
    $t> \max_i \A_{\widetilde{K}^{+}}(x_i)$ and such that
    $$
    t+2s\max (f) < \max_i \A_{\widetilde{K}^{+}}(x_i) +\eps'.
    $$
    Observe that, by \ref{ii4}, $z$ gives rise to a non-zero homology
    class in $\HF^{t+2s\max (f)}(\widetilde{K}^+)$. On the other hand,
    the composition
    $$
    \psi' \circ \psi \colon \CF^t(\widetilde{K}^+) \to \CF^{t+2s\max
      (f)}(\widetilde{K}^+)
    $$
    agrees with the inclusion map in homology. It follows that
    $z\notin \ker \psi$. This completes the proof of \ref{ii5}.
  \end{proof}

  We return now to the proof of the Lemma \ref{lem:perturbation}.
  First note that, when $\HF(x; \F)=0$, for all sufficiently small
  $s>0$, we have
\begin{equation}
\label{eq:dim_eq}
 \dim \HF^t(K^{\pm}) = \dim \HF^t(H)   
\end{equation}
for all $t\in \R$. This, for instance, follows from
\eqref{eq:constant_solution} and that $K^+ \geq H \geq K^-$.

In the rest of the proof, we assume that $\HF(x) \neq 0$.  Set
$I^-=\big(A_{K^-}(x), A_H(x)\big]$ and
$I^+=\big(\A_H(x), \A_{K^+}(x)\big]$. Then
$$
\dim \HF^t(K^{\pm}) = \const
$$
for $t \in I^{\pm}$. This is a consequence of the fact that
$\Ss(K^{\pm}) \cap I^{\pm}=\partial I^{\pm}$ (and that the intervals
$I^{\pm}$ are closed only at the right end point).  Furthermore,
\eqref{eq:dim_eq} holds for all $t \in \R- I^{\pm}$. To see this, one
can use the continuity of the associated barcodes. Alternatively, as
in \ref{ii5}, one can utilize the fact that the continuation maps
% \begin{align*}
%     &\CF(\widetilde{K}^+) \to \CF(\widetilde{H}),\\
%     &\CF(\widetilde{H}) \to \CF(\widetilde{K}^-)
% \end{align*}
\[
  \CF(\widetilde{K}^+) \to \CF(\widetilde{H}) \quad \text{and} \quad
  \CF(\widetilde{H}) \to \CF(\widetilde{K}^-)
\]
induced by linear homotopies are chain level isomorphisms and satisfy
\eqref{eq:constant_solution}.

Combining these observations, we conclude that, for sufficiently small
$s>0$, it suffices to show that at least one of the following two
inequalities holds
\begin{equation}
\label{eq:suffices}
  \dim \HF^t(K^{\pm}) \leq \dim \HF^t(H)   
\end{equation}
for some, or equivalently all, $t \in I^{\pm}$. 

Fix a sufficiently small $s>0$ and a perturbation $\tilde{H}$ as in
the claim. In what follows, we denote by $V^t(\widetilde{K}^{\pm})$
the kernel of
$$
d_{\widetilde{K}^{\pm}} \colon \CF^t(\widetilde{K}^{\pm}) \to
\CF^t(\widetilde{K}^{\pm}),
$$
where we have again suppressed the underlying almost complex structure
in the notation. Note that for $t_1 \leq t_2$ we have
$$
V^{t_1}(\widetilde{K}^{\pm}) = \{ v \in V^{t_2}(\widetilde{K}^{\pm})
\mid  \A_{\widetilde{K}^{\pm}}(v) < t_1 \}.
$$
Next, choose 
% \begin{align*}
%     t_1^{\pm} &< \min_i \A_{\widetilde{K}^{\pm}}(x_i),\\
%     t_2^{\pm} &> \max_i \A_{\widetilde{K}^{\pm}}(x_i)
% \end{align*}
\[
 t_1^{\pm} < \min_i \A_{\widetilde{K}^{\pm}}(x_i) 
\]
and
\[
t_2^{\pm} > \max_i \A_{\widetilde{K}^{\pm}}(x_i)
\]
so that all $t_i^{\pm}$ are contained in the intervals from
\ref{ii4}. We claim that
\begin{equation}
\label{eq:dim_ineq}
\dim V^{t_2^+}(\widetilde{K}^{+}) - \dim V^{t_1^+}(\widetilde{K}^{+})
\geq \dim V^{t_2^-}(\widetilde{K}^{-}) -
\dim V^{t_1^-}(\widetilde{K}^{-}). 
\end{equation}
To see this, pick a vector 
\begin{equation}
\label{eq:vector_v}
v \in V^{t_2^-}(\widetilde{K}^{-})
\setminus V^{t_1^-}(\widetilde{K}^{-}).
\end{equation}
Recall that, by \ref{ii5}, the continuation map $\psi$ is a chain
level isomorphism. Hence there exists some
$w \in V^{\infty}(\widetilde{K}^{+})$ such that $\psi(w)=v$. It
suffices to show that
$$
w \in V^{t_2^+}(\widetilde{K}^{+}) \setminus
V^{t_1^+}(\widetilde{K}^{+})
$$
or equivalently that $\A_{\widetilde{K}^{+}}(w) \in [t_1^+,
t_2^+)$. Indeed, assume that this is not the case. Then the leading
terms of $w$ belong to the set $\{y_i\}$, and we have
$$
\A_{\widetilde{K}^{+}}(w) = \A_{\widetilde{K}^{-}} (\psi(w))
=  \A_{\widetilde{K}^{-}} (v) \notin [t_1^-, t_2^-)
$$
by \eqref{eq:constant_solution}, which contradicts
\eqref{eq:vector_v}. This completes the proof of \eqref{eq:suffices}.

Finally, let us prove that \eqref{eq:dim_ineq} implies
\eqref{eq:suffices}. Set
% \begin{align*}
%   \Delta^+ &:= \dim V^{t_2^+}(\widetilde{K}^{+}) -
%              \dim V^{t_1^+}(\widetilde{K}^{+}), \\
%   \Delta^- &:=\dim V^{t_2^-}(\widetilde{K}^{-}) -
%              \dim V^{t_1^-}(\widetilde{K}^{-}). 
% \end{align*}
\[
 \Delta^+ := \dim V^{t_2^+}(\widetilde{K}^{+}) -
             \dim V^{t_1^+}(\widetilde{K}^{+}) 
\]
and
\[
  \Delta^- :=\dim V^{t_2^-}(\widetilde{K}^{-}) -
  \dim V^{t_1^-}(\widetilde{K}^{-}). 
\]
A direct computation shows that
\begin{align}
\label{eq:final1}
  \dim \HF^{t_2^-}(\widetilde{K}^-) - \dim\HF^{t_2^-}(H)
  &= 2\Delta^- - \dim \HF(x), \\
\label{eq:final2}
  \dim \HF^{t_2^+}(\widetilde{K}^+) - \dim\HF^{t_2^+}(H)
  &=   -\big(2\Delta^+ - \dim \HF(x)\big). 
\end{align}
To be more specific, \eqref{eq:final1} and \eqref{eq:final2} follow
from
\[
\dim V^{t_1^-}(\widetilde{K}^-) = \dim V^{t_1^-}(\widetilde{H}) = \dim
V^{t_2^-}(\widetilde{H})
\]
and
\[
\dim V^{t_2^+}(\widetilde{K}^+) = \dim V^{t_2^+}(\widetilde{H}) = \dim
V^{t_2^-}(\widetilde{H}),
\]
which can be, for instance, readily derived from similar statements
for the corresponding filtered homology spaces. It immediately follows
from \eqref{eq:dim_ineq} that if \eqref{eq:final1} is non-negative
then \eqref{eq:final2} is non-positive, establishing
\eqref{eq:suffices}.  This completes the proof of Lemma
\ref{lem:perturbation}. \qed

\section{Proof of Theorem \ref{thm:sh=1}}
\label{sec:pf-sh=1}
\subsection{Support of the local homology}
\label{sec:support}
In this section, we prove Part \ref{sh1} of the theorem; Part
\ref{sh2} is proved in the next section.  We start with the case where
the characteristic of $\F$ is zero, e.g., $\F=\Q$. This case is
essentially treated in \cite{CGG:Reeb-HZ}.  For the sake of
completeness, we briefly outline the argument.

First, recall that $\dim \SH(x;\F)$ is even for every field
$\F$. (This is a consequence of the fact that this is the homology of
a complex with an even number of generators.) Hence
$\dim \SH(x;\F)\geq 2$ when $\SH(x;\F)\neq 0$. If we had
$\dim \SH(x;\F)> 2$, we would also have $\dim \SH^t(W;\F)>1$ for some
$t$ close to $\CA(x)$, which is impossible. (This proves Remark
\ref{rmk:sh=1}.)

For $\F=\Q$, Part \ref{sh1} follows immediately from \eqref{eq:SH-CH},
which in turn is a consequence of the fact that the shift operator
$\DD$ in the Gysin exact sequence, \eqref{eq:loc-Gysin}, vanishes for
the local homology over $\Q$ or, more generally, when $\charr\F=0$;
\cite{GG:LS}.

Dealing with the case where $\charr \F=p$, let us assume for the sake
of simplicity that $\F=\F_p$. Then the shift operator $\DD$ need not
be zero and a different argument is needed. Let us begin with the
first part of the statement: the support of $\SH(x;\F_p)$ comprises
two consecutive degrees. This is a purely algebraic fact.  Symplectic
homology is defined over any ring and, in particular, over $\Z$, and
throughout the rest of the proof it is convenient to regard symplectic
homology groups as abelian groups rather than vector spaces. If
$\rk \SH(x; \Z)\geq 1$, the assertion follows from the case of $\F=\Q$
by the universal coefficient formula. Thus, without loss of
generality, we may assume that $\SH(x;\Z)$ is a group of finite order.

Again, by the universal coefficient formula, we have
$$
0 \to \SH_m(x;\Z)\otimes \Z_p \to \SH_m(x;\Z_p) \to
\Tor_1\big(\SH_{m-1}(x;\Z), \Z_p\big) \to 0,
$$
where $\Tor_1$ is over $\Z$. Therefore, $\SH_m(x;\Z_p)$ is non-zero if
and only if $\SH_m(x;\Z)$ or $\SH_{m-1}(x;\Z)$ has a torsion component
$\Z_k$ with $p|k$.

Let $m$ be the lowest degree such that $\SH_{m}(x;\Z_p)\neq 0$. To
prove the statement, it suffices to show that
$\SH_{m+1}(x;\Z_p)\neq 0$, and we argue as follows.

By our choice of $m$, we have $\SH_{m-1}(x;\Z_p) = 0$. Hence
$\SH_{m-1}(x;\Z)\otimes \Z_p=0$, i.e., $\SH_{m-1}(x;\Z)$ has no
torsion component $\Z_k$ with $p|k$. Then
$$
\Tor_1\big(\SH_{m-1}(x;\Z),\Z_p\big)=0.
$$
As a consequence, $\SH_m(x;\Z)$ has a torsion component $\Z_k$ with
$p|k$ since $\SH_m(x;\Z_p)\neq 0$. Then
$\Tor_1(\SH_m(x;\Z),\Z_p)\neq 0$, and thus $\SH_{m+1}(x;\Z_p)\neq 0$,
which proves the first statement in Part \ref{sh1} of the theorem for
every field $\F$.

Finally, let $q$ be the lowest degree where $\COH(x;\F)\neq 0$. Then
$\SH_q(x;\F)\neq 0$ by \eqref{eq:loc-Gysin} and $\SH_m(x;\F)= 0$
for all $m<q$. Thus $\SH(x;\F)$ is supported in degree $q$ and, as we
have just shown, also in degree $q+1$. This completes the proof of
Part \ref{sh1}.\qed

\begin{Remark}
  \label{rmk:visibility}
  As pointed out earlier, a closed Reeb orbit is $\F$-visible if and
  only if it is equivariantly $\F$-visible. This can be seen from the
  Gysin long exact sequence as follows. The sequence readily implies
  that $\SH(x;\F)=0$ whenever $\COH(x;\F)=0$, i.e., non-equivariantly
  visible closed orbits are automatically equivariantly
  visible. Conversely, assume that $x$ is equivariantly visible and
  let $m$ be the lowest degree such that $\COH_m(x;\F)\neq 0$. Then
  $\SH_m(x;\F)\neq 0$ by \eqref{eq:loc-Gysin}.
\end{Remark}

\subsection{Dynamical convexity and one cluster}
\label{sec:pf-sh2}
Let us now prove Part \ref{sh2} of Theorem \ref{thm:sh=1}.

\subsubsection{Index recurrence}
One ingredient of the proof is the fact that the Conley--Zehnder
indices of the iterates of a linear symplectic map or a finite
collection of such maps enjoy a certain recurrence property. This is
essentially a combinatorial or geometry of numbers result.

\begin{Theorem}[Index recurrence theorem, the non-degenerate case]
\label{thm:IRT}
Consider a finite collection of strongly non-degenerate elements
$\Phi_1,\dotsc,\Phi_r$ in $\TSp(2m)$ with positive mean indices. Then
for any $\eta>0$ and any $\ell_0\in\N$, there exists an integer
sequence $d_j\to\infty$ and $r$ integer sequences $k_{ij}\to\infty$,
$i=1,\dotsc, r$, such that for all $i$ and $j$, and all $\ell\in\Z$ in
the range $1\leq |\ell|\leq \ell_0$, we have
\begin{itemize}
\item[\reflb{IR1}{\rm{(IR1)}}]
  $\big|\hmu(\Phi^{k_{ij}}_i)-d_j\big|<\eta$ and, as a consequence,
  $\big|\mu(\Phi^{k_{ij}}_i)-d_j\big|\leq m$ when $\eta>0$ is small
  enough; and
 
\item[\reflb{IR2}{\rm{(IR2)}}]
  $\mu(\Phi^{k_{ij}+\ell}_i)= d_j + \mu(\Phi^\ell_i)$.
\end{itemize}
Moreover, for any $N\in \N$ we can have all $d_j$ and $k_{ij}$
divisible by~$N$.
\end{Theorem}

In particular, by \ref{IR2} arbitrarily long segments of the sequences
$\mu(\Phi_i^k)$ repeat infinitely many times up to a common
shift. Hence, the name of the theorem. We call the iterates $k_{ij}$
from the theorem an \emph{index recurrence event}.

Ultimately, Theorem \ref{thm:IRT} is a consequence of Minkowski's
theorem from geometry of numbers. It is the non-degenerate case of
either one of the following two much more general results: the
Enhanced Common Index Jump Theorem from \cite{DLW} and the Index
Recurrence Theorem from \cite{GG:LS}. (These two results are closely
related and one can obtain the latter as a consequence of the former,
although this derivation is not immediate.) Theorem \ref{thm:IRT} can
also be derived, again with some effort, from the Common Index Jump
Theorem from \cite{Lo,LZ}. A substantially more general version of the
Index Recurrence Theorem accounting for extra real parameters -- the
``actions'' of the paths $\Phi_i$ -- and also giving more information
on the sequences $d_j$ and $k_{ij}$ is proved in \cite{CGG:Reeb-HZ}.

Recall that a strongly non-degenerate ``map'' $\Phi\in\TSp(2m)$ is
said to be dynamically convex if $\mu(\Phi^k)\geq m+2$ for all
$k\in\N$. (It is easy to show that this condition is equivalent to
that $\mu(\Phi)\geq m+2$; see, e.g., \cite[Sec.\ 4.2]{GG:LS}.)
Clearly, a dynamically convex map has positive mean index. Then it is
not hard to show that Theorem \ref{thm:IRT} yields the following
corollary; see, e.g., \cite{GG:LS,CGG:Reeb-HZ}.

\begin{Corollary}
\label{cor:DC-IRT}
In the setting of Theorem \ref{thm:IRT}, assume that
$\Phi_1,\dotsc,\Phi_r$ are dynamically convex. Then for any $\eta>0$,
there exists an integer sequence $d_j\to\infty$ and $r$ integer
sequences $k_{ij}\to\infty$, $i=1,\dotsc, r$, such that for all $i$
and $j$ in addition to \ref{IR1}, we have
\begin{itemize}
\item[\reflb{IR2'}{\rm{(IR2')}}]
  $\big|\mu(\Phi^{k}_i)-d_j\big|\geq m+2$ when $k\neq k_{ij}$.
\end{itemize}
Moreover, for any $N\in \N$, we can have all $d_j$ and $k_{ij}$
divisible by~$N$.
\end{Corollary}

\subsubsection{Proof of Part \ref{sh2} of Theorem \ref{thm:sh=1}}
Throughout the proof, we will need to work simultaneously with a
ground field $\F$ of characteristic $p=0$, e.g., $\F=\Q$, and also of
characteristic $p=2$, i.g., $\F=\F_2$. Slightly modifying the notation
from Section \ref{sec:beg-end}, denote by $\PP_p$ the set of closed
Reeb orbits visible over the ground field $\F$ of characteristic
$p$. In other words, due to non-degeneracy, $\PP_p$ comprises all
closed Reeb orbits when $p=2$ and all good orbits otherwise. Thus,
$\PP_0\subset \PP_2$.  As before, let us formally add to $\PP_p$ an
extra element corresponding to the domain itself, denoting the
resulting set by $\PP'_p$. We assign zero action and index $n$ to this
element and think of it as the ``beginning'' of the unique bar
starting at action zero. Furthermore, let $\CS_p$ be the action
spectrum over $\F$ with zero added, i.e., the set of actions of the
orbits in $\PP'_p$. When the role of $p$ is immaterial, we omit it
from the notation.

In the setting of the theorem, every point in $\CS$ is the beginning
and the end of a bar, except for $0$ which is just the beginning of
the first bar. This follows readily from the long exact
sequence. Furthermore, all orbits in $\PP'$ have distinct actions.

Denote by $\CB$ or $\CB_p$ the barcode of the graded symplectic
homology persistence module over $\F$. In the proof, we will use the
maps $\beg\colon \CB\to \PP'$ (the beginning of $I$) and
$\en\colon \CB\to \PP$ from Theorem \ref{thm:beg-end}. By Remark
\ref{rmk:beg-end}, these maps are uniquely defined and, as is easy to
see, one-to-one.

Let us list the orbits in $\PP'=\{x_0,x_1,\dotsc \}$ in the order of
increasing action:
$$
0=\CA(x_0)<\CA(x_1)<\CA(x_2)<\dotsb .
$$
Likewise, we order the bars in $\CB$ as $I_0,I_1,I_2,\dotsc$, where
$x_i=\beg(I_i)$ and $x_{i+1}= \en(I_i)$. Thus we obtain strict orders
on $\PP'$ and $\CB$; and the inclusions $\PP'_0\subset \PP'_2$ and
$\CB_0\subset \CB_2$ preserve the order. In what follows, it is
convenient for our purposes to work with the sequence of the
Conley--Zehnder indices $\mu(x_i)$ starting with $i=1$, i.e., with the
first genuine closed orbit, but have the sequence of the bar degrees
$\deg(I_i)$ starting with $i=0$, i.e., with the first bar. These
sequences are related:
\begin{equation}
  \label{eq:mu-deg1}
  \textrm{ either }  \mu(x_{i})=\deg(I_i) \quad \text{or} \quad
  \mu(x_{i})=\deg(I_i)-1
\end{equation}
and
\begin{equation}
  \label{eq:mu-deg2}
  \textrm{ either }  \mu(x_{i+1})=\deg(I_i) \quad \text{or} \quad
  \mu(x_{i+1})=\deg(I_i)+1.
\end{equation}
This again follows from the long exact sequence.  Note also that
$\CB_2$ is a subdivision or a partition of $\CB_0$, i.e., in the
transition from $p=0$ to $p=2$, every bar remains unchanged or gets
partitioned into a few shorter bars.

\begin{Lemma}
  \label{lemma:degree_seq}
  Under the conditions of the theorem, we
  have:
  \begin{itemize}
  \item[\reflb{deg-0}{\rm{(a)}}] Both sequences $\mu(x_{i\geq 1})$ and
    $\deg(I_{i\geq 0})$ are increasing (possibly non-strictly).  All
    $\deg(I_i)$ have the same parity as $n$, and this sequence is a
    map onto $n-2+2\N$. The sequence $\mu(x_i)$ is strictly increasing
    at the values of $\mu(x_i)\in n- 1 + 2\N$, and for every
    $m\in n-1+2\N$ there is exactly one $x\in \PP$ with $\mu(x)=m$.
    Furthermore,
    \begin{equation}
      \label{eq:deg-0}
      \mu(x_i)=\mu(x_{i+1})\textrm{ iff }
      \deg(I_{i-1})=\deg(I_{i})=deg(I_{i+1}).
\end{equation} 

\item[\reflb{deg-1}{\rm{(b)}}] For $p=0$, the sequences
  $\mu(x_{i\geq 1})$ and $\deg(I_{i\geq 0})$ are
  $\{n+1,n+3,n+5,\dotsc\}$ and, respectively,
  $\{n,n+2,n+4,\dotsc\}$. In particular, these sequences are strictly
  increasing.

\item[\reflb{deg-2}{\rm{(c)}}] Let $x$ and $x'$ be two adjacent
  elements in $\PP_0$, i.e., good orbits such that there are no good
  orbits with action in $\big(\CA(x), \CA(x')\big)$. Then for every
  $y\in \PP_2$ with action in that range, we have
  $\mu(y)=\mu(x)+1$. Equivalently, every bar $I\in \CB_0$ either
  persists or gets partitioned in $\CB_2$ into a sequence of bars of
  degree exactly $\deg(I)$.
  \end{itemize}
  
\end{Lemma}

\begin{proof}
  By \eqref{eq:mu-deg1} and \eqref{eq:mu-deg2}, for every $i$ there
  are exactly two alternatives: either
\begin{equation}
  \label{eq:P1}
  \deg(I_i)=\mu(x_{i+1})-1 \quad \text{and} \quad \deg(I_{i+1})=\mu(x_{i+1})+1
\end{equation}
or
\begin{equation}
  \label{eq:P2}
  \deg(I_i)=\mu(x_{i+1})=\deg(I_{i+1}).
\end{equation}
Therefore, the sequence $\deg(I_i)$ is increasing with
$\deg(I_{i+1})=\deg(I_i)+2$ in the case of \eqref{eq:P1} and
$\deg(I_{i+1})=\deg(I_i)$ in \eqref{eq:P2}. As a consequence,
$\deg(I_i)$ has the same parity as $n$ since $\deg(I_0)=n$, and this
sequence is a surjective map onto $n-2+2\N$. Turning to the sequence
$\mu(x_i)$, note that $\mu(x_i)\leq \mu(x_{i+1})$ by
\eqref{eq:mu-deg1} and \eqref{eq:mu-deg2} regardless of whether
\eqref{eq:P1} or \eqref{eq:P2} holds, and
$\mu(x_{i+1})\leq \mu(x_i)+2$. Moreover, \eqref{eq:P2} cannot happen
when $\deg(x_i)$ has parity $n+1$; for then $\deg(I_i)$ would have
parity $n+1$. Hence, $\deg(x_i)$ is strictly increasing at the points
of $n-1+2\N$. It follows that for every $m$ in this set there is
exactly one $x\in \PP$ with $\deg(x)=m$.  Finally, \eqref{eq:deg-0} is
again a consequence of \eqref{eq:mu-deg1} and \eqref{eq:mu-deg2}. This
proves \ref{deg-0}.

To establish \ref{deg-1}, note that if we had $y\in \PP_0$ with
$q:=\mu(y)\in n+2\N$, we would have at least two orbits in $\PP_0$ of
degree $q-1$ or $q+1$ due to \eqref{eq:spec-CH} or, more specifically,
\eqref{eq:odd} and the condition that $\dim \SH^t(W;\Q)=1$. This is
impossible since, as we have shown, for every $m\in n-1+2\N$ an orbit
$x\in \PP_p$ with $\deg(x)=m$ is unique for any $p$ and, in
particular, $p=0$. This completes the proof of the statement that
$\mu$ on $\PP_0$ takes values in $n-1+2\N$ and is a bijection.
Together with the fact that all $\deg(I_i)$ have parity $n$, this also
completely rules out \eqref{eq:P2}. Therefore, the degree map $\deg$
on $\CB_0$ takes values in $n-2+2\N$ and is also a bijection, proving
\ref{deg-1}. Finally, \ref{deg-2} is an immediate consequence of
\ref{deg-0} and \ref{deg-1}.
\end{proof}

Dynamical convexity of the Reeb flow readily follows from the
lemma. By \ref{deg-0}, $\mu(x_i)\geq \mu(x_1)$ for
$\{x_1,x_2,\dotsc\}=\PP_2$. Since $x_1$ is the lowest-action closed
Reeb orbit, it is necessarily simple and hence $x_1\in\PP_0$. Now, by
\ref{deg-1}, $\mu(x_1)=n+1$. Therefore, $\mu(x_i)\geq n+1$ for all
$x_i\in\PP_2$, i.e., the flow is dynamically convex.

To prove the rest of the statement, fix $r\geq 1$ and let
$Y=\{y_1,\dotsc,y_r\}$ be the first $r=q$ prime orbits in $\PP_2$,
ordered arbitrarily. If the flow has finitely many prime orbits and
$r$ is greater than their number, we relabel $r$ as that number and
take the set of all prime orbits as $Y$. Let
$\{y_1,\dotsc,y_q\}=Y\cap\PP_0$ be the non-alternating orbits in
$Y$. We need to prove that $q\leq n$, and to establish that all orbits
have the same action/mean index ratio it suffices to prove this for
the orbits in $Y$.

Consider an index recurrence event for the linearized flows of the
orbits from $Y$, where we require $d_j$ and all $k_{ij}$ to be even,
$d_j$ to be sufficiently large, and $\eta<1$.  By \ref{IR1} from
Theorem \ref{thm:IRT},
$$
\big|\mu(y_i^{k_{ij}})-d_j\big|\leq n-1.
$$
Let $\CI:=\Z\cap [d_j-n+1,d_j+n-1]$. This interval has exactly $n$
points of parity $n+1$ since $d_j$ is even. Note that
$y_i^{k_{ij}}\in\PP_0$ for $i\leq q$. Thus the orbits $y_i^{k_{ij}}$
have distinct indices by \ref{deg-1} and have at most $n$ available
slots to fill. (Some of the slots can be occupied by the iterates of
prime orbits that are not in $Y$.) Hence, $q\leq n$.

Next, consider the consecutive bars from $\CB_0$ of degree in the
range from $d-n-2$ to $d+n$. By \ref{deg-1}, there are exactly $n+1$
such bars. Since by Theorem \ref{thm:depth} all bars have length
bounded by some constant $\Cbar$, the total length of these bars is
bounded from above by $(n+1)\Cbar$. The actions $\CA(y_i^{k_{ij}})$
appear as some of the end-points of these bars for $i\leq q$ but not
for $i>q$. Due to \ref{deg-2}, the actions $\CA(y_i^{k_{ij}})$ for
$i>q$ partition these bars from $\CB_0$ into the bars in
$\CB_2$. Thus, for any $i\leq r$ and $l\leq r$, we have
$$
\big|k_{ij}\CA(y_i)-k_{lj}\CA(y_l)\big|
=\big|\CA(y_i^{k_{ij}})-\CA(y_l^{k_{lj}})\big|\leq (n+1)\Cbar
$$
and, by \ref{IR1},
$$
\big|k_{ij}\hmu(y_i)-k_{lj}\hmu(y_l)\big|
=\big|\hmu(y_i^{k_{ij}})-\hmu(y_l^{k_{lj}})\big|\leq 2\eta.
$$
Dividing the first of these inequalities by $\CA(y_l)$, we have
$$
\big|k_{ij}\CA(y_i)/\CA(y_l)-k_{lj}\big|
\leq (n+1)\Cbar/\CA(y_l),
$$
which, combined the second inequality, yields
$$
\big|k_{ij}\hmu(y_i)-k_{ij}\hmu(y_l)\CA(y_i)/\CA(y_l)\big|
\leq C,
$$
where the constant $C$ is independent of $j$. Dividing by $k_{ij}$ and
passing to the limit as $k_{ij}\to\infty$, we obtain
$\hmu(y_i)=\hmu(y_l)\CA(y_i)/\CA(y_l)$, i.e.,
$\hmu(y_i)/\CA(y_i)=\hmu(y_l)/\CA(y_l)$. This concludes the proof of
the theorem.  \hfill\qed

\end{document}